\documentclass[twoside,a4paper]{amsart}
\usepackage{stmaryrd}
\usepackage{ amssymb }
\usepackage{mathabx}
\usepackage{pdfsync}
\usepackage{xcolor}
\definecolor{webblue}{rgb}{0,.5,0}
\definecolor{webblue}{rgb}{.6,0,0}
\definecolor{RoyalBlue}{cmyk}{1, 0.50, 0, 0}
\usepackage[colorlinks=true, breaklinks=true, urlcolor=webblue, linkcolor=RoyalBlue, citecolor=webblue,backref=page]{hyperref}

\usepackage{courier} 

\usepackage{epsfig, graphicx, subfigure}
\usepackage{verbatim, setspace}
\usepackage{amsmath, amssymb}

\def\cal{\mathcal}

\usepackage[pdftex]{pict2e}

\newcommand{\T}     {\mathbb{T}}

\newcommand{\R}     {\mathbb{R}}
\newcommand{\N}     {\mathbb{N}}
\newcommand{\Z}     {\mathbb{Z}}

\newcommand{\dist}{\mathrm{dist}}

\newcommand{\supp}{\mathrm{supp}}

\newcommand{\e} {{\boldsymbol e}}

\def\cal{\mathcal}

\def\ge{\geqslant}
\def\le{\leqslant}

\newtheorem{theorem}{Theorem}[section]

\newtheorem{corollary}[theorem]{Corollary}
\newtheorem{lemma}[theorem]{Lemma}

\theoremstyle{remark}

\numberwithin{equation}{section}

\begin{document}

\title[ Wave packet decomposition for Schr\"odinger evolution with rough potential \ldots]{Wave packet decomposition for Schr\"odinger evolution with rough potential and generic value of parameter}
\author{Sergey A. Denisov}

\thanks{
This research was supported by the grant NSF-DMS-2054465  and by the Van Vleck Professorship Research Award. 
}

\address{
\begin{flushleft}
Sergey Denisov: denissov@wisc.edu\\\vspace{0.1cm}
University of Wisconsin--Madison\\  Department of Mathematics\\
480 Lincoln Dr., Madison, WI, 53706,
USA\vspace{0.1cm}\\
\end{flushleft}
}\maketitle

\begin{abstract}
We develop the wave packet decomposition to study the Schr\"odinger evolution with rough potential. As an application, we obtain the improved bound on the wave propagation for the generic value of parameter.
\end{abstract} \vspace{1cm}



\maketitle

\setcounter{tocdepth}{3}

\tableofcontents

\section{Introduction.}

That work aims to develop a new technique in perturbation theory for dispersive equations. As a model case, we take Schr\"odinger evolution with time-dependent real-valued potential $V(x,t)$:
\begin{equation}\label{e1}
iu_t(x,t,k)=-k\Delta u(x,t,k)+V(x,t)u(x,t,k), \quad u(x,0,k)=f(x),\quad t\in \R\,\quad \Delta:=\partial^2_{xx}\,,
\end{equation}
where $k\in \R$ is  a real-valued parameter. For a large value of $T\gg 1$, we consider the problem \eqref{e1} either on the real line $x\in \R$ or on the torus $x\in \R/2\pi T\Z$ and make the following assumptions about the real-valued potential $V$ and the initial data $f$ when studying $u$ for $t\in [0,2\pi T]$:\smallskip

\noindent {\bf (A)} if  $x\in \R$, then
\begin{eqnarray}
 \|V(x,t)\|_{L^\infty(\R\times [0,2\pi T])}\lesssim  T^{-\gamma}, \gamma>0\,,
\\
\|f(x)\|_{2}<\infty, \;
\end{eqnarray}

\noindent {\bf (B)} if  $x\in \R/2\pi T\Z$, then
\begin{eqnarray}
 \|V(x,t)\|_{L^\infty(\R/2\pi T\Z\times [0,2\pi T])}\lesssim  T^{-\gamma}, \gamma>0\,,
\\
\|f(x)\|_{2}<\infty\,.
\end{eqnarray}
The behavior of free evolution $e^{i\Delta t}$ is well-understood (the presence of $k$ only scales the time $t$). It is governed by the dispersion relation for which the higher Fourier modes propagate  with higher speed, giving rise to ballistic transport.  
In the current work, we study how the presence of $V$ changes the free evolution $e^{ik\Delta}f$ for the time interval $t\in [0,2\pi T]$.

The solution $u$ in \eqref{e1} is understood as  the solution to the Duhamel integral equation
\begin{equation}\label{duha}
u(\cdot,t,k)=e^{ik\Delta t}f-i\int_0^t e^{ik\Delta (t-\tau)
}V(\cdot,\tau)u(\cdot,\tau,k)d\tau, \quad t\in [0,2\pi T]
\end{equation}
in the class $u\in C([0,2\pi T],L^2)$. Its existence and uniqueness are immediate from the contraction mapping principle. More generally,  for $0\le t_1\le t\le 2\pi T$, the symbol $U(t_1,t,k)$ denotes the operator 
$f\mapsto U(t_1,t,k)f$ where $U(t_1,t,k)f$ solves
\begin{equation}\label{evce1}
i\partial_t U(t_1,t,k)f=-k\Delta U(t_1,t,k)f+V(x,t)U(t_1,t,k)f, \quad U(t_1,t_1,k)f=f\,.
\end{equation}
The propagator $U$ satisfies the standard group property: $U(t_1,t_1,k)=I, U(t_1,t_2,k)U(t_0,t_1,k)=U(t_0,t_2,k)$, and, since $V$ is real-valued, $U$ is a unitary operator in $L^2: \|Uf\|_{2}=\|f\|_2$. To set the stage, we start with an elementary perturbative result (below, the symbol $\|O\|$ indicates the operator norm of the operator $O$ in the space $L^2$).
\begin{lemma} \label{l1} Suppose $0\le t_1\le t\le t_2\le 2\pi T$ and $\|V\|_{\infty}\lesssim T^{-\gamma}$. Then, 
\begin{equation}\label{duh1}
U(t_1,t,k)=e^{ik\Delta (t-t_1)}-i\int_{t_1}^t  e^{ik\Delta (t-\tau)
}V(\cdot,\tau)  e^{ik\Delta (\tau-t_1)
}  d\tau+Err
\end{equation}
and the operator norm of $Err$ allows the estimate
$
 \|Err\|\lesssim  T^{-2\gamma}(t_2-t_1)^2\,.
$
\end{lemma}
\begin{proof}That follows from the representation
\begin{eqnarray}\label{duz1}
U(t_1,t,k)=e^{ik\Delta (t-t_1)}-i\int_{t_1}^t e^{ik\Delta (t-\tau)
}V(\cdot,\tau)e^{ik\Delta (\tau-t_1)}d\tau+Err\,,\\
Err:=-\int_{t_1}^t e^{ik\Delta (t-\tau_1)
}V(\cdot,\tau_1)\int_{t_1}^{\tau_1} e^{ik\Delta (\tau_1-\tau_2)
}V(\cdot,\tau_2)U(t_1,\tau_2,k)d\tau_2d\tau_1\nonumber
\end{eqnarray}
and two bounds: $\|V\|_{L^\infty}\lesssim  T^{-\gamma}$ and $\|U(t_1,\tau_2,k)\|=1$.
\end{proof}
That lemma has an immediate application. 
\begin{corollary}Suppose $\gamma>1$, then
\begin{equation}\label{asa1}
\|U(0,t,k)-e^{ik\Delta t}\|=O(T^{1-\gamma}), \, t\in [0,2\pi T], \, k\in \R\,.
\end{equation}\label{pl2}
\end{corollary}
Hence, for $\gamma>1$, the propagated wave is asymptotically close to the free dynamics irrespective of the value of $k$ and of the initial vector $f$. 
In the current paper, we study the problem for generic $k$ and show that the analogous result holds for some $\gamma$ below the elementary threshold $1$. We also give examples which show that \eqref{asa1} cannot hold for $\gamma<1$ for all $k$, in general.

 The Schr\"odinger evolution with smooth $V$  was studied in \cite{bourgain1,erdogan,wang} where the upper and lower estimates for the Sobolev norms of the solution were obtained. 
 In the context of the general evolution equations, these questions were addressed in \cite{d1,d0}. Equations similar to \eqref{e1} appear in the study of  Green's function $G(\cdot,\cdot,k^2)$ of the stationary  two-dimensional Schr\"odinger equation $-\Delta+V$ with slowly-decaying $V(x_1,x_2)$ where $(x_1,x_2)\in \R^2$  and $k\in \R^+$. When written in the polar coordinates $(\theta,r), \theta\in [0,2\pi), r>0$ with $r$ considered as ``time''-variable $t$, the WKB-type correction takes the form close to \eqref{e1} 
  (see \cite{d2} for more detail). The asymptotics of solutions to the $2\times 2$ systems of ODE  for generic value of parameters was studied using the harmonic analysis methods (see, e.g., \cite{ck1,ck2} and references therein). Below, we develop a different approach. It is based on writing the linear in $V$ term in \eqref{duz1} 
 \begin{equation}\label{mt1}
 \int_{t_1}^t e^{ik\Delta (t-\tau)
}V(\cdot,\tau)e^{ik\Delta (\tau-t_1)}d\tau
 \end{equation}
 in the convenient basis of wave packets (see, e.g., \cite{guth}) which allows a detailed and physically appealing treatment of the multi-scattering situation at hand.
 We will mostly focus on problem {\bf (A)}; however, our analysis is valid for the problem {\bf (B)}, too. Denote (see Figure 1)
\begin{equation}\label{lk2}
\Upsilon_T:=\{x\in [-\pi T,\pi T], t\in  [cT,2\pi T], 0<c<2\pi\}\,.
\end{equation}
The following theorem is one application of our perturbative technique.

\begin{theorem}\label{tg1} There is $\gamma_0\in (0,1)$ such that the following statement holds. 
Suppose $V$ satisfies
\begin{eqnarray}\label{usl1}
\|V\|_{L^\infty(\Upsilon^c_T)}\lesssim T^{-\alpha},\alpha>1,\\ \|V\|_{L^\infty(\Upsilon_T)}\lesssim T^{-\gamma}, \gamma>\gamma_0\,,\\
\label{ryr1}
 \supp\, \widehat V(\xi_1,\xi_2)\subset \{\xi: \rho_1<|\xi|<\rho_2\}
 \end{eqnarray} and $\rho_2>\rho_1>0$. Then, for each interval $I\subset \R^+$, there is a positive parameter $\delta(\rho_1,\rho_2,I)$ such that for every function $f$ that satisfies 
\begin{eqnarray}
\|f\|_2=1,\\
\label{asd1}
\|f\|_{L^2(|x|>2\pi T)}=o(1), \, T\to\infty
\end{eqnarray} and 
\begin{equation}\label{asd2}
\supp \,\widehat f \subset (-\delta,\delta)\,,
\end{equation} there exists a set $Nres\subset I$ such that 
$\lim_{T\to\infty} |Res|=0, Res:=I\backslash Nres$ and
\begin{equation}\label{lop2}
\lim_{T\to\infty}  \|u(x,2\pi T,k)-e^{ik \Delta (2\pi T) }f\|_2\to 0
\end{equation}
for $k\in Nres$.
\end{theorem}

\begin{picture}(200,140)
\put(100,50){\textcolor{yellow}{\rule{220\unitlength}{50\unitlength}}}
\put(160,5){Figure 1}
\put(20,25){\line(1,0){350}}
\put(20,75){\vector(1,0){350}}
\put(380,65){$t$}
\put(325,65){$2\pi T$}
\put(20,125){\line(1,0){350}}
\put(20,15){\vector(0,1){120}}
\put(-10,22){$-2\pi T$}
\put(-3,122){$2\pi T$}
\put(10,135){$x$}
\put(320,15){\line(0,1){120}}
\put(100,50){\line(0,1){50}}
\put(100,100){\line(1,0){220}}
\put(100,50){\line(1,0){220}}
\put(180,60){$\Upsilon_T$}
\end{picture}

\noindent {\bf Remark.} In contrast to Corollary \ref{pl2}, the theorem says that the propagation is asymptotically free even for some $\gamma<1$ if we consider a generic value of $k$. It is crucial that we can handle essentially arbitrary initial data $f$.  In fact, for $f$ that is well-localized on the frequency side, a simple perturbative argument can be used to get an analogous result. The conditions \eqref{asd1}, \eqref{asd2}, and \eqref{usl1}-\eqref{ryr1} on the $f$ and $V$  are included to guarantee that the bulk of the wave $u$ given by $f$ hits $\Upsilon_T$.
On the other hand, assumption \eqref{asd2} is essential for the statement to hold and cannot be dropped as will be illustrated in Section~\ref{s6} by  example.

\smallskip

\noindent {\bf Remark.} For given $f$, such set $Nres$ (which can depend on $f$) will be called the set of {\it non-resonant} parameters within the interval $I$. Later, for every $\gamma<1$, we will present $f$ and potential $V$  so that the resonant set $Res$ is nonempty. \smallskip

\noindent {\bf Remark.} Simple modification of Lemma~\ref{l1} shows that the potential $V$ from Theorem \ref{tg1}  is negligible outside $\Upsilon_T$ due to \eqref{usl1}. Inside $\Upsilon_T$ it satisfies the ``weak decay condition'' if $\gamma<1$, and it oscillates on scale $\sim 1$ there due to \eqref{ryr1}. For the problem we consider, the assumption \eqref{ryr1} is actually necessary since dropping it can introduce a well-known WKB-type correction in the dynamics and \eqref{lop2} would fail.

\begin{lemma} \label{lk5} Let $\widehat\phi(\xi_1,\xi_2)$ be a smooth centrally symmetric real-valued bump function supported on $B_1(0)$ in $\R^2$. Denote $\kappa:=\sqrt T$ and let $\phi_\kappa(x,t)=\phi(x/\kappa,t/\kappa)$. Then, the function
\begin{eqnarray}\nonumber
V=
\sum_{\stackrel{|2\pi n\kappa|<T/4,}{|2\pi m \kappa-\pi T|<T/20}} T^{-\gamma}c_{n,m}\cos(x\lambda_{n,m}+t\mu_{n,m})\phi_\kappa(x-2\pi n\kappa,t-2\pi m\kappa),\\ (n,m)\in \mathbb{Z}^2,\, c_{n,m}\in [-1,1], \gamma>\gamma_0, \lambda_{n,m}^2+\mu_{n,m}^2\sim 1\label{hsh1}
\end{eqnarray}
 satisfies conditions of the previous theorem.
 \end{lemma}
 
 {\it Proof.} Indeed, $\widehat \phi\in \mathbb{S}(\R^2)$ so $ \phi\in \mathbb{S}(\R^2)$ and $\phi_\kappa$ is real-valued. Then, the Fourier transform of $\cos(x\lambda_{n,m}+t\mu_{m,n})\phi_\kappa(x-2\pi n\kappa,t-2\pi m\kappa)$ is supported on two balls centered at $(\pm \lambda_{n,m},\pm \mu_{n,m})$ with radii $\kappa^{-1}$. Moreover, 
 \[
 |V|\le C_b T^{-\gamma}\sum_{|2\pi n\kappa|<T/4,|2\pi m\kappa-\pi T|<T/20} \Bigl(1+|x\kappa^{-1}-2\pi n|+|t\kappa^{-1}-2\pi m| \Bigr)^{-b}
 \]
with any positive $b$. Hence, taking $c$ small enough in the definition of $\Upsilon_T$, we get $|V|\lesssim T^{-\gamma}$ on $\Upsilon_T$ and $|V|\le C_b T^{-b}$ with any positive $b$ outside of $\Upsilon_T$. The condition $\lambda_{n,m}^2+\mu_{n,m}^2\sim 1$ guarantees that the function $\widehat V$ is supported on the annulus $\rho_1<|\xi|<\rho_2$ with some positive $\rho_1$ and $\rho_2$.\qed\smallskip

The random potentials $V$ that model Anderson localization/delocalization phenomenon  exhibit strong oscillation (see, e.g., \cite{bourgain,den,yao,safronov}) similar to the one considered in Theorem \ref{tg1}. \smallskip

The structure of the paper is as follows. The second section contains the formulation of our central perturbation result, we develop the suitable wave packet decomposition there. The third and fourth sections frame the problem in this new setup, they contain the main arguments and the proof of Theorem~\ref{tg1}. Section five provides some examples of resonance formation. In the appendix, we collect some auxiliary results.\bigskip


{\bf Some notation:}\smallskip

$\bullet\quad$ If $x\in \R^d$ and $r>0$, then $B_r(x)$ stands for the closed ball of radius $r$ around the point $x$. Given a set $E\subset \R^d$, the symbol $E^c$ denotes its complement $E^c=\R^d\backslash E$.

$\bullet\quad$ The Fourier transform in $\R^d$ of function $f$ is denoted by 
\[
\mathcal{F}f=\widehat f=(2\pi)^{-\frac d2}\int_{\R^d} f e^{- i \langle x,\xi\rangle }dx\,.
\]
We will use both symbols\, $\widehat\cdot$\, and $\mathcal{F}$ in the text,  depending on what is more convenient.

$\bullet\quad$ Suppose $f_T\in L^2(\R), \|f_T\|_{L^2(\R)}=1$ and $T\gg 1$ is a large parameter. We will say that $f_T$ is concentrated around $a_T$ at scale $\ell_T$ ($a_T\in \R, \ell_T>0$) if
\[
\lim_{q\to +\infty} \limsup_{T\to \infty}\int_{|x-a_T|>q\ell_T}|f_T|^2dx=0\,.
\]

$\bullet\quad$ Let $\kappa:=T^{\frac 12}$. We  introduce the lattice $2\pi \kappa \Z\times 2\pi \kappa  \Z$ and call the cubes $\{B_{p,q}\}=[2\pi \kappa  p, 2\pi \kappa  (p+1)]\times [2\pi \kappa  q, 2\pi \kappa  (q+1)], p,q\in \Z$ the characteristic cubes in $(x,t)\in \R^2$. We will say that a point has ``unit coordinates'' $(p,q)$ if its coordinate $x=2\pi \kappa p$ and its coordinate $t=2\pi q\kappa$.

$\bullet\quad$ If $B$ is a cube on the plane (or an interval on $\R$), then $c_B$ denotes its center and $\alpha B$ denotes the $\alpha$-dilation of $B$  around $c_B, \alpha>0$.

$\bullet\quad$ Given an interval $I=[k_1,k_2]\subset \R^+$, the symbol $I^{-1}$ denotes $\{\eta=k^{-1}\,|\,k\in I\}=[k_2^{-1},k_1^{-1}]$.

$\bullet\quad$ Symbol $C^\infty_c(\R^d)$ denotes the set of smooth compactly supported functions on $\R^d$ and $\mathbb{S}(\R^d)$ stands for the Schwartz class of functions.


$\bullet\quad$ We will use the following notation standard in the modern harmonic analysis, i.e., given two positive quantities $A$ and $B$ and a large parameter $T$, we write $A\lessapprox B$ if there is $C_\epsilon$ such that
$
A\le C_\epsilon T^\epsilon B
$
for all $T>1$ and all $\epsilon>0$. Given a quantity $f$, the symbol $O(f)$ will indicate another quantity that satisfies $|O(f)|\le C|f|$ with a constant $C$.

$\bullet\quad$ For a function $g\in L^2(\R)$ and a measurable set $E\subset \R$, the symbol $P_{E}g$ will indicate the Fourier orthogonal projection to the set $E$, i.e., $P_Eg=\mathcal{F}^{-1}(\mathcal{F} g \cdot \chi_E)$.

$\bullet\quad$  Symbol $\kappa$ is used for $T^{\frac 12}$ and $\eta$ for $k^{-1}$. In many estimates below, the letter $R$ will indicate a positive quantity that satisfies $R\lessapprox 1$ where $\kappa$ and $T$ are  large parameters.
\bigskip

\section{The main estimate and the wave packet decomposition.}

We start with a result that illustrates the Fourier restriction phenomenon of the Duhamel operator $\int_0^Te^{ik(T-t)\Delta}dt$.
\begin{lemma}Suppose $g(x,t)\in \mathbb{S}(\R^2)$. Then,
\[
\int_\R\left\|\int_0^{2\pi T}e^{ik\Delta (2\pi T-t)}\partial_xg(\cdot,t)dt\right\|_{2}^2 dk\lesssim \|g\|^2_{L^2(\R^2)}
\]  
and
\[
\int_\R\left\|P_{|\xi|>\delta}\int_0^{2\pi T}e^{ik\Delta (2\pi T-t)}g(\cdot,t)dt\right\|_{2}^2 dk\lesssim \delta^{-2} \|g\|^2_{L^2(\R^2)}
\]
for all $\delta>0$.\label{lbbb}
\end{lemma}
\begin{proof}
Taking the Fourier transform in variable $x$, we get
\[
\mathcal{F}\left(\int_0^{2\pi T} e^{-ik\Delta t}\partial_xg(\cdot,t)dt\right)=C\xi\int_0^{2\pi T}  e^{ikt\xi^2} (\mathcal{F} g)(\xi,t)dt=C\xi\mathcal{F}( \chi_{0<t<2\pi T}\cdot g)(\xi,-k\xi^2),
\]
where  $\mathcal{F}$ in the right-hand side is a two-dimensional Fourier transform, which is  restricted to a $k$-dependent parabola. The simple change of variables and Plancherel theorem gives us
\[
\int_\R\int_\R |\xi\mathcal{F}( \chi_{0<t<2\pi T}\cdot g)(\xi,-k\xi^2)|^2d\xi dk\lesssim \int_{\R}\int_{0<t<2\pi T} |g(x,t)|^2dx dt\le \|g\|_2^2\,.
\]
The second bound in the lemma follows immediately from the first.
\end{proof}
In our paper, we will be estimating expression (the one-collision operator)
\[
\int_0^{2\pi T} e^{ik\Delta (2\pi T-t)}V(\cdot,t)e^{ik\Delta t}fdt\,,
\]
which corresponds to the linear in $V$ term in \eqref{duh1}. Notice that $V(\cdot,t)e^{ikt\Delta}f$ depends on $k$ and the previous lemma is not applicable. However, for $f$ localized on low frequencies on the Fourier side, the dependence on $k$ is so weak that a variant of that lemma can be used to obtain an estimate better than the general bound we get. To treat the general $f$, we will apply the different technique standard in the harmonic analysis. That approach uses the decomposition of $f$ into a different ``basis'' of the so-called wave packets and each element of such basis has a particular localization both on the physical and on the Fourier side. Such a decomposition is a standard tool in modern harmonic analysis (see, e.g., \cite{guth} for one application).

\subsection{$k$-dependent wave packet decomposition.} 

In the proof of Theorem \ref{tg1}, we can take for $I$ the finite union of dyadic intervals $[2^{j},2^{j+1}), j\in \Z$ and, therefore, it is enough to prove the claim for dyadic intervals only. Now, notice that given $k\in [2^{j},2^{j+1})$, we can rescale the space variable $\widetilde x=2^{(j-1)/2} x$
to reduce the problem to $k\in [2,4)$. The support of $\widehat V$ (parameters $\rho_1$ and $\rho_2$) and the domain $\Upsilon_T$  will change, too, but the proof can be easily adapted.
Let $I=[2,4]$. Take a nonnegative function $h\in C^\infty_c(\R)$ supported on $[0,4\pi]$ such that the partition of unity 
\begin{equation}\label{loy1}
1=\sum_{n\in \Z}h^2(s-2\pi n)
\end{equation}
holds for $s\in \R$. Since $k\in I$, we get $k\ge 2$. Then, for every $f\in \mathbb{S}(\R)$ and every $s\in [2\pi n,2\pi(n+k)]$, we can write
\[
f(s)h(s-2\pi n)=(2\pi k)^{-1}\sum_{\ell\in \Z}  \left(\int_{2\pi n}^{2\pi(n+k)}f(\tau)h(\tau-2\pi n)e^{- i\ell k^{-1} (\tau-2\pi n)} d\tau\right) e^{ i\ell k^{-1} (s-2\pi n)}, 
\]
and,  if we introduce an auxiliary variable $\eta=k^{-1}$,
\begin{eqnarray*}
f(s)\stackrel{\eqref{loy1}}{=}\sum_{n,\ell\in \Z^2} \widetilde f_{n,\ell}(\eta) e^{ i\ell \eta s}h(s-2\pi n), \quad \widetilde f_{n,\ell}(\eta):=\eta(2\pi)^{-1}\int_{2\pi n}^{2\pi(n+k)}f(\tau)h(\tau-2\pi n)e^{- i\ell \eta \tau} d\tau\\\stackrel{k>2,\,\supp\, h\,\subset [0,4\pi]}{=}\eta(2\pi)^{-1}\int_{2\pi n}^{2\pi (n+2)}f(\tau)h(\tau-2\pi n)e^{- i\ell \eta \tau} d\tau\,.
\end{eqnarray*}
We dilate by $\kappa$ to write
\begin{equation}\label{wpack1}
f(x)=\sum_{n,\ell\in \Z^2} \omega_{n,\ell}(x,k)f_{n,\ell}(\eta), \, 
\omega_{n,\ell}(x,k)= h_n(x)e^{ i\ell \eta \kappa^{-1}x}\,,
\end{equation}
where
\begin{equation}\label{pico1}
h_n(x):=\kappa^{-\frac 12}h\left(\frac{x-2\pi \kappa n}{\kappa}\right), \, 
f_{n,\ell}(\eta):=\eta(2\pi)^{-1}\int_{2\pi n\kappa}^{2\pi (n+2)\kappa}f(x)h_n(x)e^{- i\ell \eta\kappa^{-1} x} dx
\end{equation}
and $f_{n,\ell}(\eta)$ can be written as follows
\begin{eqnarray}\nonumber
f_{n,\ell}(\eta)=e^{- 2\pi i\eta n\ell}f^*_{n,\ell}(\eta), \,\\ f^*_{n,\ell}(\eta):=\frac{\eta}{ 2\pi \sqrt\kappa}\int_{0}^{4\pi \kappa}f(s+2\pi n\kappa)h(s\kappa^{-1})e^{- i\ell \eta s\kappa^{-1}} ds\,. \label{lil1}
\end{eqnarray}
By Plancherel theorem, 
\[
\sum_{\ell\in \Z} |f_{n,\ell}^*|^2=\frac{\eta}{2\pi} \int_\R | f(s)h(s\kappa^{-1}-2\pi n)|^2ds
\]
and
\begin{equation}\label{ss1}
\sum_{n,\ell\in \Z^2} |f_{n,\ell}^*|^2\sim  \|f\|_2^2\,,
\end{equation}
because $\sum_{n\in \Z}|h(x-2\pi n)|^2=1$ and $\eta \in I^{-1}$. By the standard approximation argument, we can extend \eqref{ss1} from $f\in \mathbb{S}(\R)$ to $f\in L^2(\R)$.\smallskip

Recall that $\eta=k^{-1}$. Using the known evolution of $\omega_{0,0}$ for $t\in [0,2\pi T]$ (check \eqref{ffa1}), we can use the modulation property $(a)$ from Appendix (check \eqref{isad1}) to get
\begin{equation}\label{klop1}
e^{ik\Delta t}\omega_{n,\ell}=\kappa^{-\frac 12}\Omega_{T_{n,\ell}}(x,t,k)e^{ i\eta(\alpha(T^\rightarrow) x-\alpha^2(T^\rightarrow) t)},\quad  \alpha(T^\rightarrow):=\ell/\kappa\,,
\end{equation}
where the function 
\[\Omega_{T_{n,\ell}}(x,t,k)=\cal U((x-2\pi(n+1)\kappa-2t\alpha(T^\rightarrow))/\kappa,tk/\kappa^2)\] oscillates  slowly  in  $x\in \R$ and $t\in [-T,T]$, i.e.,
\[
\left|\frac{\partial^j \Omega_{T_{n,\ell}}(x,t,k)}{\partial x^j}\right|\le_j \kappa^{-j}, \, \left|\frac{\partial^j \Omega_{T_{n,\ell}}(x,t,k)}{\partial t^j}\right|\le_j \kappa^{-j}\,,
\]
provided $|\ell|\lesssim \kappa$,
and its derivatives in $k$ are bounded as follows
\begin{equation}\label{fdf1}
\left|\frac{\partial^j \Omega_{T_{n,\ell}}(x,t,k)}{\partial k^j}\right|\le_j 1, \, k\in I\,.
\end{equation}
Moreover, $\Omega_{T_{n,\ell}}(x,t,k)$ is negligible away from $\{|x-2\pi (n+1)\kappa-2t\alpha(T^\rightarrow)|<\kappa^{1+\delta}, |t|<2\pi T\}$ where $\delta$ is any positive fixed parameter. It will be convenient to work with tubes 
\[T^\rightarrow=T_{n,\ell}:=\{|x-2\pi (n+1)\kappa-2t\alpha(T^\rightarrow)|<2\pi\kappa, \, |t|<2\pi T\}
\] 
that are $k$-independent. Such $T^\rightarrow$ has base (when $t=0$) as an interval $x\in [2\pi n\kappa,2\pi(n+2)\kappa]$ and $2\alpha(T^\rightarrow)$ is its slope in variable $t$. Given $T^\rightarrow$, we denote the corresponding $n$ as $n(T^\rightarrow)$. Notice that $T_{n,\ell}\cap T_{n+j,\ell}=\emptyset$ for $|j|>1$. Combining \eqref{wpack1} and \eqref{klop1}, we obtain the formula for the free evolution of $f$
\begin{equation}\label{isad28}
e^{ik\Delta t}f= \sum_{T^\rightarrow} f^*_{T^\rightarrow}(\eta)\cdot \Bigl(\kappa^{-\frac 12}\Omega_{T^\rightarrow}(x,t,k)\Bigr)\cdot e^{ i\eta(\alpha(T^\rightarrow)(x-2\pi n(T^\rightarrow)\kappa)-\alpha^2(T^\rightarrow) t)}\,,
\end{equation}
where we used $T^\rightarrow$ as the index of summation to account for all $n$ and $\ell$.

 One can be more specific about $\Omega_{T_{n,\ell}}(x,t,k)$: if $\mathcal{U}$ is  a function generated by $h$ in \eqref{gene1} then \eqref{trans1} gives
\begin{equation}\label{pty1}
\Omega_{T_{n,\ell}}(x,t,k)=\sum_{\lambda\in \Z} \Omega_{T_{n,\ell}}^{(\lambda)}(x,t,k),\quad \Omega_{T_{n,\ell}}^{(\lambda)}(x,t,k):=\mathcal{U}_\lambda((x-2\pi (n+1)\kappa)/\kappa-2t\alpha(T^\rightarrow)/\kappa^2,tk/\kappa^2)\,.
\end{equation}
Notice that  $\mathcal{U}_0((x-2\pi (n+1)\kappa)/\kappa-2t\alpha(T^\rightarrow)/\kappa^2, tk/\kappa^2)$ is supported inside the $T_{n,\ell}$ and all other terms $\mathcal{U}_\lambda((x-2\pi (n+1)\kappa)/\kappa-2t\alpha(T^\rightarrow)/\kappa^2, tk/\kappa^2), \lambda\neq 0$ are supported in the tubes obtained by its vertical translations by $2\pi \kappa \lambda$. Although the supports of all terms in \eqref{pty1} are parallel tubes, the contributions from large $\lambda$ are negligible due to \eqref{pio1}. Hence, we can rewrite the formula for evolution \eqref{isad28} in a slightly different form
\begin{equation}\label{isad2}
e^{ik\Delta t}f=\sum_\lambda \sum_{T^\rightarrow} f^*_{T^\rightarrow}(\eta)\cdot \Bigl(\kappa^{-\frac 12}\Omega^{(\lambda)}_{T^\rightarrow}(x,t,k)\Bigr)\cdot e^{ i\eta(\alpha(T^\rightarrow)(x-2\pi n(T^\rightarrow)\kappa)-\alpha^2(T^\rightarrow) t)}\,.
\end{equation}
We now discuss the properties of the coefficients $f_{n,\ell}$ in \eqref{pico1}. Clearly,
\[
f_{n,\ell}(\eta)=\frac{\eta}{(2\pi)^{3/2}}(\widehat f\ast \widehat h_n)(\ell\eta/\kappa),\, \widehat h_n(\xi)=\sqrt\kappa e^{-2\pi i\kappa n\xi}\widehat h(\xi \kappa)\,.
\]
Hence, conditions \eqref{asd1} and \eqref{asd2} imply that
\[
\sum_{|n|\ge 2\kappa} \sum_{\ell\in \Z} |f_{n,\ell}|^2=o(1), \quad
\sum_{|n|\le 2\kappa} \sum_{|\ell|\ge C \delta\kappa}|f_{n,\ell}|^2=o(1)\,,
\]
where $T\to\infty$, $\eta\in I^{-1}$ and $C$ is a positive constant. Since the evolution $U$ in \eqref{evce1} is unitary, that indicates that, when studying $Uf$, we can restrict our attention to wave packets $\omega_{n,\ell}$ with $|n|\le 2\kappa$ and $|\ell|\le C \delta\kappa$ where the last condition is equivalent to $|\alpha(T^\rightarrow)|\lesssim \delta$.

The expression in \eqref{mt1} will be central for our study and the following theorem is the key for proving Theorem \ref{tg1}. Recall that $\kappa:=T^{\frac 12}, \eta:=k^{-1}$.

\begin{theorem}[The main estimate for {\it ``one-collision operator''}]\label{tg2}  Suppose $V$ and $f$ satisfy assumptions of the Theorem \ref{tg1}, except that $V$ is not necessarily real-valued. Define \label{io1}
\begin{equation}\label{kgh1}
f_o=\sum_{|n|\le 2\kappa, |\ell|\le C\delta\kappa} \omega_{n,\ell}(x,k)f_{n,\ell}(\eta)\,
\end{equation}
and
\begin{equation}\label{coll}
Q=\int_0^{2\pi T} e^{ik\Delta (2\pi T-t)}V(\cdot,t)e^{ik\Delta t}dt\,.
\end{equation}
Then, 
\begin{equation}\label{mesy1}
\left(\int_{I^{-1}}\|Qf_o\|^2d\eta\right)^{\frac 12}\lessapprox \kappa^{-\frac 14}T^{1-\gamma}
\end{equation}
for sufficiently small parameter $\delta$.
\end{theorem}
We call the operator $Q$ from \eqref{coll} the {\it one-collision operator.}
The proof of this theorem will be split into several statements to be discussed in the next sections. We will finish  that proof in Section~\ref{s4}.\bigskip

\section{Interaction of wave packets with $V$ on characteristic cubes.}\label{s3}

Consider the characteristic cubes $B_{p,q}=[2\pi p\kappa, 2\pi(p+1)\kappa]\times [2\pi q\kappa, 2\pi (q+1)\kappa],\, p,q\in \mathbb{Z}^+\,.$
We split the time interval $[0,2\pi T]$ into $\kappa$ equal intervals of length $2\pi\kappa$ so that
\[
[0,2\pi T]=\bigcup_{q=1}^{\kappa}[2\pi\kappa (q-1),2\pi\kappa q]
\]
and $\Upsilon_T$ is covered by $\sim \kappa^2$ characteristic cubes.
Recall that we study the operator 
\[
Q=\int_{0}^{2\pi T} e^{ik \Delta (2\pi T-\tau)
}V(\cdot,\tau)e^{ik\Delta \tau}d\tau
\]
from \eqref{coll}. We will do that by first taking the partition of unity 
\begin{equation}\label{pino1}
1=\sum_{p,q\in \Z^2}\phi(x/(2\pi \kappa)-p,t/(2\pi\kappa)-q),
\end{equation} where smooth $\phi=1$ on  $0.9\cdot ([0,1]\times [0,1])$ and is zero outside $1.1\cdot ([0,1]\times [0,1])$. If $B$ is a characteristic cube with parameters $(p,q)$, we might  use notation $V_{B}$ instead of $V_{p,q}$. Let 
\begin{equation}\label{ivan1}
V=\sum_{p,q} V_{p,q}(x-2\pi \kappa p,t-2\pi \kappa q), V_{p,q}:=V(x+2\pi p\kappa,t+2\pi q\kappa)\cdot \phi(x/(2\pi\kappa),t/(2\pi\kappa)).
\end{equation}
Later, we will need the following lemma which shows that each $V_{p,q}$ shares the main properties of $V$ but it is localized to $1.1B_{0,0}$ instead of $\Upsilon_T$.
\begin{lemma}\label{leg6} Given assumptions of Theorem \ref{tg2}, we get
\begin{eqnarray}\label{kyta1}
\supp\, V_{p,q}\subset 1.1B_{0,0},\\ \label{ryr19}
 \|V_{p,q}\|_{L^\infty(\R^2)}\lesssim T^{-\gamma}\,,\\
\label{ryr8}
 \|\widehat V_{p,q}(\xi_1,\xi_2)\|_{(L^\infty\cap L^1)( \{\xi: \rho_1-\epsilon<|\xi|<\rho_2+\epsilon\}^c)}\le_{\epsilon,N}T^{-N}
 \end{eqnarray} for every small $\epsilon>0$ and every $N\in \mathbb{N}$.
\end{lemma}
\begin{proof}The first two bounds are immediate and the last one follows from the properties of convolution.
\end{proof}
Later in the text, we will use the following observation several times.
Since $V_{p,q}$ is supported on the ball of radius $C\kappa$, we can write
$V_{p,q}=V_{p,q}\cdot \rho(x/(C_1\kappa),t/(C_1\kappa))$ where $\rho$  a smooth bump function which vanishes outside $B_2(0)$ and is equal to one on $B_{1}(0)$. The constant $C_1$ is taken large enough. Then, 
\begin{equation}\label{niko0}
\widehat V_{p,q}(\xi)=\int_{\R^2} \widehat V_{p,q}(s)D_\kappa(\xi-s)ds, \quad D_\kappa(s):=(C_1\kappa)^2\widehat\rho(C_1\kappa s),\quad \widehat\rho\in \mathbb{S}(\R^2)\,,
\end{equation}
i.e., $\widehat V_{p,q}$ can be written as a convolution of $\widehat V_{p,q}$ with a function well-localized at scale $\kappa^{-1}$ around the origin. Hence, applying the Cauchy-Schwarz inequality, we get
\begin{equation}\label{niko1}
|\widehat V_{p,q}(\xi)|^2\le \int_{\R^2} |\widehat V_{p,q}(s)|^2|D_\kappa(\xi-s)|ds\cdot  \int_{\R^2} |D_\kappa(\xi-s)|ds\lesssim \int_{\R^2} |\widehat V_{p,q}(s)|^2|D_\kappa(\xi-s)|ds\,.
\end{equation}
Informally, that shows that the value of function $|\widehat V_{p,q}|^2$ at any point $\xi$ can be estimated by, essentially, the average of that function on $\kappa^{-1}$-ball around that point plus an error coming from the fast-decaying tail of $\widehat \rho$.

\subsection{Contribution from each characteristic cube $B_{p,q}$.}

We first focus on the contribution to scattering picture coming from $V_{p,q}(x-2\pi \kappa p,t-2\pi \kappa q)$. Notice that in the expression 
\[
\int_{0}^{2\pi T} e^{ik\Delta (2\pi T-\tau)
}V_{p,q}(\cdot-2\pi \kappa p,\tau-2\pi \kappa q)e^{ik\Delta \tau}f_od\tau
\]
the function $V_{p,q}(\cdot-2\pi \kappa p,\tau-2\pi \kappa q)e^{ik\Delta \tau}f_o$ satisfies
\begin{equation}\label{gfa1}
\left\|P_{|\xi|>C}\Bigl(V_{p,q}(\cdot-2\pi \kappa p,\tau-2\pi\kappa q)e^{ik\Delta \tau}f_o\Bigr)\right\|_2\le_{N}T^{-N},\, N\in \N
\end{equation}
for suitable $C$ due to \eqref{ryr8}.

We apply wave packet decomposition for $f_o$ and for $Qf_o$. The wave packets based on $t=0$ will be denoted $\omega^{\rightarrow}_{n,\ell}$ and those corresponding to $t=2\pi T$ will be denoted $\omega^{\leftarrow}_{m,j}$.  
Similarly, the corresponding tubes are $T^{\rightarrow}_{n,\ell}$ and $T^{\leftarrow}_{m,j}$. Recall that, given any forward tube $T^{\rightarrow}$, we denote the corresponding wave packet $\omega^\rightarrow_{n,\ell}$ as $\omega(T^\rightarrow)$, the corresponding coordinate $n=n(T^\rightarrow)$ and the frequency $\ell(T^\rightarrow)=\kappa\alpha(T^\rightarrow)$. The notation for the backward tube is similar. The set of forward tubes relevant to us is 
\begin{equation}\label{lil2}
\mathcal{T}^{\rightarrow}=\{T^\rightarrow: |n(T^\rightarrow)|\le 2\kappa, |\alpha(T^\rightarrow)|\le C\delta\}
\end{equation}
and, thanks to \eqref{gfa1}, the set of backward tubes  of interest  satisfies
$ |n(T^\leftarrow)|\le C\kappa, |\alpha(T^\leftarrow)|\le C$.
From now on, we will take only these tubes into account.
  In estimating $Qf_o$, we will use wave packet decomposition and the bound \eqref{ss1}.

 \begin{picture}(200,150)
\put(160,5){Figure 2}
\put(100,50){\textcolor{yellow}{\rule{220\unitlength}{50\unitlength}}}
\put(20,25){\line(1,0){350}}
\put(325,65){$2\pi T$}
\put(20,125){\line(1,0){350}}
\put(20,15){\vector(0,1){120}}
\put(-10,22){$-2\pi T$}
\put(-3,122){$2\pi T$}
\put(10,135){$x$}
\put(320,15){\line(0,1){120}}
\put(100,50){\line(0,1){50}}
\put(100,100){\line(1,0){220}}
\put(100,50){\line(1,0){220}}
\put(220,60){$\Upsilon_T$}
\put(40,95){$T^\rightarrow$}
\put(260,110){$T^\leftarrow$}
\put(125,70){\line(1,0){60}}
\put(125,60){\line(1,0){60}}
\put(125,50){\line(1,0){60}}
\put(125,80){\line(1,0){60}}
\put(145,50){\line(0,1){40}}
\put(155,50){\line(0,1){40}}
\put(165,50){\line(0,1){40}}
\put(175,50){\line(0,1){40}}
\put(20,90){\vector(5,-1){270}}
\put(20,85){\color{brown}\line(5,-1){250}}
\put(20,95){\color{brown}\line(5,-1){250}}
\put(320,120){\vector(-3,-1){248}}
\put(80,45){\color{brown}\line(3,1){240}}
\put(80,35){\color{brown}\line(3,1){240}}
\linethickness{0.3mm}
\put(145,70){\color{red}\line(1,0){10}}
\put(145,60){\color{red}\line(1,0){10}}
\put(145,60){\color{red}\line(0,1){10}}
\put(155,60){\color{red}\line(0,1){10}}
\put(150,51){$B_{p,q}$}
\end{picture}

The southwestern corner of the cube $B_{p,q}$ has coordinates $2\pi\kappa p, 2\pi\kappa q$ and, if it is intersected by forward tube $T^\rightarrow$ and backward tube $T^\leftarrow$, then their parameters satisfy relations (see Figure 2)
\begin{equation}\label{resd1}
2\alpha(T^\rightarrow) q=p-n(T^\rightarrow)+O(1), \quad 2\alpha(T^\leftarrow)(\kappa-q)=n(T^\leftarrow)-p+O(1)
\end{equation}
and $O(1)$ indicates a real-valued quantity which depends only on  $p,q,T^\leftarrow,T^\rightarrow$ and satisfies a uniform estimate $|O(1)|<C$.

Then, taking $f_o$, applying $Q$, and using \eqref{wpack1} and \eqref{isad2}, we compute the coordinate with respect to backward tube $T^{\leftarrow}$ by the formula (we suppress summation in $\lambda$ here for a moment)
\begin{eqnarray}\nonumber
\langle Q_{p,q}f_o, \omega(T^{\leftarrow})\rangle=\hspace{11cm}\\\nonumber\sum_{T^\rightarrow\in \mathcal{T}^{\rightarrow}} f_{T^\rightarrow}(\eta)\int_0^{2\pi T}   \langle V_{p,q}(x-2\pi\kappa p,t-2\pi\kappa q)  e^{ik\Delta \tau}\omega(T^\rightarrow)
,e^{-ik\Delta (2\pi T-\tau)} \omega(T^\leftarrow)\rangle d\tau=\\
\nonumber
\sum_{T^\rightarrow\in \mathcal{T}^{\rightarrow}} f_{T^\rightarrow}^*(\eta)\int_{\R^2}
\Bigl( V_{p,q}(x-2\pi\kappa p,\tau-2\pi \kappa q) \kappa^{-\frac 12}\Omega_{T^\rightarrow}(x,\tau,k)e^{ i\eta(\alpha(T^\rightarrow)(x-2\pi n(T^\rightarrow)\kappa)-\alpha^2(T^\rightarrow)\tau)}\\  \times\quad  \kappa^{-\frac 12}
\overline{\Omega_{T^\leftarrow}(x,\tau-2\pi T,k)}e^{- i\eta(\alpha(T^\leftarrow)x-\alpha^2(T^\leftarrow)(\tau-2\pi T))} \Bigr)dx d\tau \label{pogk}
=\\\nonumber
\Bigl(
\sum_{T^\rightarrow\in \mathcal{T}^{\rightarrow}} f_{T^\rightarrow}^*(\eta)\int_{\R^2}
\Bigl( V_{p,q}(x-2\pi\kappa p,\tau-2\pi \kappa q) \kappa^{-\frac 12}\Omega_{T^\rightarrow}(x,\tau,k)e^{ i\eta(\alpha(T^\rightarrow)(x-2\pi n(T^\rightarrow)\kappa)-\alpha^2(T^\rightarrow)\tau)} \\ 
 \times\quad  \kappa^{-\frac 12}
\overline{\Omega_{T^\leftarrow}(x,\tau-2\pi T,k)}e^{- i\eta(\alpha(T^\leftarrow)(x-2\pi n(T^\leftarrow)\kappa)-\alpha^2(T^\leftarrow)(\tau-2\pi T))} \Bigr)dx d\tau\Bigr)e^{-i\eta\alpha(T^\leftarrow)2\pi n(T^\leftarrow)\kappa} \,.\nonumber
\end{eqnarray}

If the tubes $T^\leftarrow$ and $T^\rightarrow$ both intersect $B_{p,q}$, then \eqref{resd1} holds and the integral above can be written as \[\exp\Bigl(2\pi i\eta \kappa \Bigl((\alpha^2(T^\rightarrow)-\alpha^2(T^\leftarrow)) q +\alpha^2(T^\leftarrow)\kappa+O(1)\Bigr)\Bigr) F_{T^\rightarrow,T^\leftarrow}(\eta),\] where
\begin{eqnarray}\nonumber F_{T^\rightarrow,T^\leftarrow}(\eta):=
\int_{\R^2} \Bigl( V_{p,q}(x-2\pi\kappa p,\tau-2\pi\kappa q) \Bigl(\kappa^{-1}\Omega_{T^\rightarrow}(x,\tau,k)\overline{\Omega_{T^\leftarrow}(x,\tau-2\pi T,k)}\Bigr)\\e^{ i\eta((\alpha(T^\rightarrow)-\alpha(T^\leftarrow))(x-2\pi p\kappa)-(\alpha^2(T^\rightarrow)-\alpha^2(T^\leftarrow))(\tau-2\pi q\kappa))}\Bigr) dx d\tau\,.\label{eff1}
\end{eqnarray}

Notice that the formula \eqref{pty1} can be applied to $\Omega_{T^\rightarrow}(x,\tau,k)$ and $\Omega_{T^\leftarrow}(x,\tau-2\pi T,k)$ which gives
\begin{eqnarray*}
V_{p,q}(x-2\pi\kappa p,\tau-2\pi \kappa q)\Omega_{T^\rightarrow}(x,\tau,k)\overline{\Omega_{T^\leftarrow}(x,\tau-2\pi T,k)}=\hspace{6cm}\\\hspace{2cm}\sum_{\lambda,\lambda'\in \Z^2} V_{p,q}(x-2\pi\kappa p,\tau-2\pi \kappa q)\Omega_{T^\rightarrow}^{(\lambda)}(x,\tau,k)\cdot \overline{\Omega_{T^\leftarrow}^{(\lambda')}(x,\tau-2\pi T,k)}
\end{eqnarray*}
and, accordingly, 
\begin{equation}\label{laos1}
F_{T^\rightarrow,T^\leftarrow}(\eta)=\sum_{\lambda,\lambda'\in \Z^2}F^{(\lambda,\lambda')}_{T^\rightarrow,T^\leftarrow}(\eta)\,.
\end{equation}
Now, $\Omega_{T^\rightarrow}^{(\lambda)}(x,\tau,k)\cdot \overline{\Omega_{T^\leftarrow}^{(\lambda')}(x,\tau-2\pi T,k)}$ is compactly supported inside the intersection of corresponding tubes and its sup-norm decays fast in $|\lambda|$ and $|\lambda'|$ due to \eqref{pio1}. In all estimates that follow, we will only handle the term that corresponds to $\lambda=\lambda'=0$, i.e.
\begin{eqnarray}\nonumber
\langle Q_{p,q}^{(0,0)}f_o, \omega(T^{\leftarrow})\rangle
= e^{2\pi i\eta(-\kappa \alpha(T^\leftarrow) n(T^\leftarrow)+\kappa^2\alpha^2(T^\leftarrow))} \times \\
\label{reb1}
\sum_{T^\rightarrow: T^\rightarrow\cap B_{p,q}\neq \emptyset} f^*_{T^\rightarrow}(\eta)
e^{ 2\pi i\eta \kappa ((\alpha^2(T^\rightarrow)-\alpha^2(T^\leftarrow)) q +O(1))} F^{(0,0)}_{T^\rightarrow,T^\leftarrow}(\eta)\,,\end{eqnarray}
where the backward tube $T^\leftarrow$ intersects $B_{p,q}$. 
The other terms in \eqref{laos1} lead to the same bounds except that the resulting estimates will involve a strong decay in $|\lambda|$ and $|\lambda'|$. The first factor $e^{2\pi i\eta(-\kappa \alpha(T^\leftarrow) n(T^\leftarrow)+\kappa^2\alpha^2(T^\leftarrow))}$ in the formula above  only depends on $T^\leftarrow$ and it is unimodular. It will play no role in our estimates. Now, $\Omega^{(0)}_{T^\rightarrow}(x,t,k)$ is supported inside $T^\rightarrow$ and we consider
\begin{eqnarray}\nonumber
F^{(0,0)}_{T^\rightarrow,T^\leftarrow}(\eta)=
\int_{\R^2} \Bigl( V_{p,q}(x-2\pi\kappa p,\tau-2\pi\kappa q) \Bigl(\kappa^{-1}\Omega^{(0)}_{T^\rightarrow}(x,\tau,k)\overline{\Omega^{(0)}_{T^\leftarrow}(x,\tau-2\pi T,k)}\Bigr)\times \\e^{ i\eta((\alpha(T^\rightarrow)-\alpha(T^\leftarrow))(x-2\pi\kappa p)-(\alpha^2(T^\rightarrow)-\alpha^2(T^\leftarrow))(\tau-2\pi\kappa q))}\Bigr)dx d\tau\ \label{ker2}=\\\nonumber
\int_{\R^2} \Bigl( V_{p,q}(x,\tau) \Bigl(\kappa^{-1}\Omega^{(0)}_{T^\rightarrow}(x+2\pi\kappa p,\tau+2\pi\kappa q,k)\overline{\Omega^{(0)}_{T^\leftarrow}(x+2\pi\kappa p,\tau+2\pi\kappa q-2\pi T,k)}\Bigr)\times \\e^{ i\eta((\alpha(T^\rightarrow)-\alpha(T^\leftarrow))x-(\alpha^2(T^\rightarrow)-\alpha^2(T^\leftarrow))\tau)}\Bigr)dx d\tau\,.\nonumber
\end{eqnarray}
By  \eqref{pty1}, \eqref{pino1}, and \eqref{pio1}, the function $\phi(x/(2\pi\kappa)- p,t/(2\pi\kappa)- q)\Omega^{(0)}_{T^\rightarrow}(x,\tau,k)\overline{\Omega^{(0)}_{T^\leftarrow}(x,\tau-2\pi T,k)}$ is smooth and is supported on $CB_{p,q}$. Arguing like in \eqref{niko0},  we can write an estimate
\begin{eqnarray}\label{gas1}
|F^{(0,0)}_{T^\rightarrow, T^\leftarrow}(\eta)|\lesssim 
\kappa^{-1}
\Bigl(|\widehat V_{p,q}|\ast | \psi_{\kappa}|\Bigr)(-(\alpha(T^\rightarrow)-\alpha(T^\leftarrow))\eta, (\alpha^2(T^\rightarrow)-\alpha^2(T^\leftarrow))\eta)\,,
\end{eqnarray}
where $\psi_\kappa(\xi_1,\xi_2)=\kappa^2\Psi(\kappa\xi_1,\kappa\xi_2)$, $\Psi\in \mathbb{S}(\R^2)$, and $\Psi$ does not depend on $p,q,T^\rightarrow,T^\leftarrow$ and $k\in I$.  Consider the function $|\widehat V_{p,q}|\ast | \psi_{\kappa}|$. Applying the Cauchy-Schwarz inequality to 
 $|\widehat  V_{p,q}|\ast |\psi_\kappa|$ we get
 \begin{eqnarray}\label{tri2}
 (|\widehat  V_{p,q}|\ast |\psi_\kappa|)^2(\xi)\le \left(\int_{\R^2}|\widehat  V_{p,q}(s)|^2\ast |\psi_\kappa(\xi-s)|ds\right)\cdot \int_{\R^2} |\psi_\kappa(\xi-s)|ds\lesssim \\ \int_{\R^2}|\widehat  V_{p,q}(s)|^2\ast |\psi_\kappa(\xi-s)|ds\,,\nonumber
 \end{eqnarray}
which is the analog of \eqref{niko1}. From the Lemma \ref{leg6}, we conclude that the function $|\widehat V_{p,q}|\ast | \psi_{\kappa}|$ satisfies a bound (check \eqref{ryr8})
\begin{equation}\label{fife1}
\|\Bigl(|\widehat V_{p,q}|\ast | \psi_{\kappa}|\Bigr)\|_{(L^\infty\cap L^1)( \{\xi: \rho_1-\epsilon<|\xi|<\rho_2+\epsilon\}^c)}\le_{\epsilon,N}T^{-N}
 \end{equation} for every small $\epsilon>0$ and every $N\in \mathbb{N}$.
Moreover,
\begin{equation}\label{eey1}
\|\Bigl(|\widehat V_{p,q}|\ast | \psi_{\kappa}|\Bigr)\|_{2}\lesssim T^{\frac 12-\gamma}\,.
 \end{equation}
Formula \eqref{gas1} shows that $F^{(0,0)}_{T^\rightarrow, T^\leftarrow}(\eta)$ is bounded by the restriction of 
$|\widehat V_{p,q}|\ast | \psi_{\kappa}|$ to the ``curve'':
\[(-(\alpha(T^\rightarrow)-\alpha(T^\leftarrow))\eta, (\alpha^2(T^\rightarrow)-\alpha^2(T^\leftarrow))\eta),\quad \eta\in I^{-1}\,.
\]
For fixed $\alpha(T^\rightarrow)$ and $\alpha(T^\leftarrow)$, that is a straight segment or a point $0$ when $\eta\in I^{-1}$. For fixed $\eta\in I^{-1}$ and $\alpha(T^\rightarrow)$, that is a piece of parabola when $\alpha(T^\leftarrow)$ spans a segment.
By \eqref{fife1}, $F^{(0,0)}_{T^\rightarrow, T^\leftarrow}$ is negligible unless (recall that $\eta\in [\frac 14,\frac 12]$)
\begin{equation}\label{san1}
|\alpha(T^\rightarrow)-\alpha(T^\leftarrow)|\sim_{\rho_1,\rho_2} 1, \,|\alpha(T^\rightarrow)+\alpha(T^\leftarrow)|\sim_{\rho_1,\rho_2} 1\,.
\end{equation}
For given $\rho_1$ and $\rho_2$, we can choose positive $\delta$ small enough to guarantee that $|\alpha(T^\rightarrow)|\le \delta$ implies
\begin{equation}\label{ssa1}
|\alpha(T^\leftarrow)|\sim 1\,.
\end{equation}
Without loss of generality, we can therefore assume that $\alpha(T^\leftarrow)\sim 1$.
Hence, the set of backward tubes relevant to us is \[\mathcal{T}^{\leftarrow}:=\{T^\leftarrow: |n(T^\leftarrow)|<C(\rho_1,\rho_2,\delta)\kappa,\,\, 0<C_1(\rho_1,\rho_2,\delta)\le \alpha(T^\leftarrow)\le C_2(\rho_1,\rho_2,\delta)\}\]
so that \eqref{san1} holds for each $T^\rightarrow\in \mathcal{T}^{\rightarrow}$ and $T^\leftarrow\in \mathcal{T}^{\leftarrow}$. 
Notice that each $B_{p,q}\in \Upsilon_T$ is intersected by $\sim \delta \kappa$ tubes  $T^\rightarrow\in \mathcal{T}^\rightarrow$ and by $\sim \kappa$ tubes $T^\leftarrow\in \mathcal{T}^\leftarrow$.

\subsection{Sparsifying $V$.} \label{spaspa} For  $T^\rightarrow\in \mathcal{T}^\rightarrow$ and $T^\leftarrow\in \mathcal{T}^\leftarrow$,  each set $T^\rightarrow\cap T^\leftarrow$ can be covered by at most $C$ cubes $B_{p,q}$ (see Figure  2 above). 

 \begin{picture}(200,120)
\put(160,5){Figure 3}
\put(125,70){\line(1,0){90}}
\put(125,60){\line(1,0){90}}
\put(125,50){\line(1,0){90}}
\put(125,80){\line(1,0){90}}
\put(125,30){\line(1,0){90}}
\put(125,40){\line(1,0){90}}
\put(125,90){\line(1,0){90}}
\put(125,100){\line(1,0){90}}
\put(145,20){\line(0,1){90}}
\put(155,20){\line(0,1){90}}
\put(165,20){\line(0,1){90}}
\put(175,20){\line(0,1){90}}
\put(205,20){\line(0,1){90}}
\put(135,20){\line(0,1){90}}
\put(185,20){\line(0,1){90}}
\put(195,20){\line(0,1){90}}
\linethickness{0.5mm}
\put(135,20){\line(0,1){90}}
\put(165,20){\line(0,1){90}}
\put(195,20){\line(0,1){90}}
\put(125,30){\line(1,0){90}}
\put(125,60){\line(1,0){90}}
\put(125,90){\line(1,0){90}}
\put(145,60){\textcolor{red}{\rule{10\unitlength}{10\unitlength}}}
\put(145,90){\textcolor{red}{\rule{10\unitlength}{10\unitlength}}}
\put(175,60){\textcolor{red}{\rule{10\unitlength}{10\unitlength}}}
\put(175,90){\textcolor{red}{\rule{10\unitlength}{10\unitlength}}}
\put(145,30){\textcolor{red}{\rule{10\unitlength}{10\unitlength}}}
\put(175,30){\textcolor{red}{\rule{10\unitlength}{10\unitlength}}}
\put(205,30){\textcolor{red}{\rule{10\unitlength}{10\unitlength}}}
\put(205,60){\textcolor{red}{\rule{10\unitlength}{10\unitlength}}}
\put(205,90){\textcolor{red}{\rule{10\unitlength}{10\unitlength}}}
\end{picture}

Notice that, for each $P\in \mathbb{N}$, the potential $V$ can be written (by ``sparsifying'' periodically) as 
\begin{equation}\label{sparse}
V(x,t)=\sum_{p,q\in \Z^2}V_{p,q}(x-2\pi \kappa p,t-2\pi \kappa q)=\sum_{\alpha\in \{0,\ldots,P-1\},\beta\in \{0,\ldots,P-1\}} V^{(\alpha,\beta)}(x,t)\,,
\end{equation}
where
each  $V^{(\alpha,\beta)}$, defined by
\[
V^{(\alpha,\beta)}(x,t):=\sum_{n,m\in \Z^2}V_{nP+\alpha,mP+\beta}(x-2\pi\kappa (nP+\alpha), t-2\pi\kappa (mP+\beta))\,,
\]
satisfies a bound
 $
 \dist \,(\,\supp\, V_{nP+\alpha,mP+\beta},\supp\, V_{n'P+\alpha,m'P+\beta})\gtrsim P\kappa
 $
 for all $(n,m)\neq (n',m')$. Figure~3 has $P=3$ and the red boxes correspond to choosing $\alpha=0$ and $\beta=1$.
 
   Therefore, for $P$ large enough, $T^\rightarrow\in \mathcal{T}^\rightarrow$ and $T^\leftarrow\in \mathcal{T}^\leftarrow$, the set  $T^\rightarrow\cap T^\leftarrow$ intersects at most one characteristic cube out of the set $B_{nP+\alpha,mP+\beta}$ when $\alpha$ and $\beta$ are fixed.
   
    Hence, going from one function $V$ to any of  $P^2$ functions $V^{(\alpha,\beta)}$, we can always guarantee that {\it each set  $T^\rightarrow\cap T^\leftarrow$ intersects at most one
$B_{p,q}$}. Going forward, we can assume without loss of generality that this property holds for the original $V$.
\smallskip

\subsection{Contribution from all of $V$.}
In the general case, given $T^\leftarrow$, we need to account for all characteristic cubes that intersect it. Hence, we need to control
\begin{eqnarray}\label{lhl1}
D_{T^\leftarrow}(\eta):= e^{2\pi i\eta(-\kappa \alpha(T^\leftarrow) n(T^\leftarrow)+\kappa^2\alpha^2(T^\leftarrow))}\times \hspace{4cm}\\
\sum_{B_{p,q}: B_{p,q}\cap T^{\leftarrow}\neq \emptyset}\quad \sum_{T^\rightarrow: B_{p,q}\cap T^{\rightarrow}\neq \emptyset} f_{T^\rightarrow}^*(\eta)e^{2\pi i\eta \kappa \left((\alpha^2(T^\rightarrow)-\alpha^2(T^\leftarrow)) q +O(1)\right) } F^{(0,0)}_{T^\rightarrow,T^\leftarrow}(\eta)\,.\nonumber
\end{eqnarray}
Take $\mu(\eta)$, any smooth non-negative bump function supported on $I^{-1}$ and write
\begin{eqnarray}\label{lhl2}
\int_{\R} \mu(\eta)\left|D_{T^\leftarrow}(\eta)\right|^2d\eta=\int_{\R} \mu(\eta)\times \hspace{5cm}\\ 
\sum_{B_{p,q}\cap T^{\leftarrow}\neq \emptyset }
\sum_{B_{p',q'}\cap T^{\leftarrow}\neq \emptyset}
\sum_{\stackrel{T^\rightarrow: }{B_{p,q}\cap T^{\rightarrow}\neq \emptyset}  }
\sum_{\stackrel{ T'^{\rightarrow}:}{ B_{p',q'}\cap T'^{\rightarrow} \neq \emptyset}  }
e^{2\pi i\eta \kappa \left((\alpha^2(T^\rightarrow)-\alpha^2(T^\leftarrow)) q -(\alpha^2(T'^\rightarrow)-\alpha^2(T^\leftarrow)) q'+O(1)\right)}
\nonumber
\\ \nonumber \times\,\,
f_{T^\rightarrow}^*(\eta)\overline{f_{T'^\rightarrow}^*(\eta)}
F^{(0,0)}_{T^\rightarrow,T^\leftarrow}(\eta)\overline{F^{(0,0)}_{T'^\rightarrow,T^\leftarrow}(\eta)}d\eta=\Sigma^{(nr)}_{T^\leftarrow}+\Sigma^{(r)}_{T^\leftarrow}\,,
\end{eqnarray}
where $\Sigma^{(r)}_{T^\leftarrow}$ is the sum that corresponds to the resonance indexes given by the following {\it global resonance conditions:}
\begin{equation}\label{global}\,\,
 \left|\left((\alpha^2(T^\rightarrow)-\alpha^2(T^\leftarrow)) q -(\alpha^2(T'^\rightarrow)-\alpha^2(T^\leftarrow)) q' \right)\right|\lessapprox 1\,,
\end{equation}
that do not depend on the choice of $V$. For definiteness, we can assume that the bound $\lessapprox 1$ takes the form: $|\cdot |=O(R), R= C_{\varepsilon} \kappa^{\varepsilon}$ with a fixed positive $\varepsilon$ that can be chosen arbitrarily small. Having mentioned that, we define $\mathcal{G}(T^\rightarrow, T'^\rightarrow)$ as a set of backward tubes $T^\leftarrow$ that form the global resonance with $T^\rightarrow$ and $T'^\rightarrow$.

\begin{lemma}We have $|\Sigma^{(nr)}_{T^\leftarrow}|\le C_j\kappa^{-j}$ for every $j\in \mathbb{N}$.\label{reso}
\end{lemma}
\begin{proof}
Recall that \eqref{kgh1} restricts $|\ell|\le C\delta \kappa$ in \eqref{lil1}.  Next, one can apply a non-stationary phase argument. In the integral
\[
\int_{\R} \mu
f_{T^\rightarrow}^*\overline{f_{T'^\rightarrow}^*}e^{2\pi i\eta \kappa \left((\alpha^2(T^\rightarrow)-\alpha^2(T^\leftarrow)) q -(\alpha^2(T'^\rightarrow)-\alpha^2(T^\leftarrow)) q' \right)}e^{2\pi i\kappa \eta O(1)}F^{(0,0)}_{T^\rightarrow,T^\leftarrow}\overline{F^{(0,0)}_{T'^\rightarrow,T^\leftarrow}}d\eta\,,
\]
we use \eqref{lil1}, \eqref{eff1}, and the bound \eqref{fdf1} to integrate in $\eta$ by parts consecutively to get
\begin{eqnarray*}
\left|\int_{\R} \mu
f_{T^\rightarrow}^*\overline{f_{T'^\rightarrow}^*}e^{2\pi i\eta \kappa \left((\alpha^2(T^\rightarrow)-\alpha^2(T^\leftarrow)) q -(\alpha^2(T'^\rightarrow)-\alpha^2(T^\leftarrow)) q' \right)}e^{2\pi i\kappa \eta O(1)}F^{(0,0)}_{T^\rightarrow,T^\leftarrow}\overline{F^{(0,0)}_{T'^\rightarrow,T^\leftarrow}}d\eta\right|\lesssim \\
\le_{j,\varepsilon}\kappa^{-j}, j\in \mathbb{N}, \varepsilon>0\,,
\end{eqnarray*}
provided that
$
\left|\left((\alpha^2(T^\rightarrow)-\alpha^2(T^\leftarrow)) q -(\alpha^2(T'^\rightarrow)-\alpha^2(T^\leftarrow)) q' \right)\right|\ge \kappa^\varepsilon\,.
$
Since the quadruple sum in \eqref{lhl2}  contains at most $C\kappa^4$ terms, we get the statement of the lemma.
\end{proof}
This lemma shows that we only need to focus on the resonant terms which constitute a small proportion of all possible combinations. We will study these resonant configurations next.\smallskip

\section{Resonant terms and proofs of the main results.} \label{s4} For the generic $\eta\in I^{-1}$, our goal is to estimate $\sum_{T^\leftarrow}|D_{T^\leftarrow}(\eta)|^2$ where $D_{T^\leftarrow}$ is from \eqref{lhl1}. By Lemma~\ref{reso},
 the contribution from non-resonant indexes is negligible. Consider the sum of the resonant terms. It can be rewritten in the following form 
\begin{eqnarray}\label{kook1}
A:=\sum_{T^\leftarrow} \quad \widehat{\sum_{T^\rightarrow,T'^\rightarrow }} f^*_{T^\rightarrow}\bar f^*_{T'^\rightarrow}\times\hspace{6cm} \\  \nonumber e^{2\pi i\eta \kappa \left((\alpha^2(T^\rightarrow)-\alpha^2(T^\leftarrow)) q -(\alpha^2(T'^\rightarrow)-\alpha^2(T^\leftarrow)) q' \right)}
 e^{2\pi i\kappa \eta O(1)}F^{(0,0)}_{T^\rightarrow,T^\leftarrow}\overline{F^{(0,0)}_{T'^\rightarrow,T^\leftarrow}} \\
   =\sum_{T^\rightarrow,T'^\rightarrow } f^*_{T^\rightarrow}\bar f^*_{T'^\rightarrow} \times \hspace{6cm}\label{kook2}\\{\sum_{T^\leftarrow\in \mathcal{G}(T^\rightarrow,T'^\rightarrow)}}\nonumber   e^{2\pi i\eta \kappa \left((\alpha^2(T^\rightarrow)-\alpha^2(T^\leftarrow)) q -(\alpha^2(T'^\rightarrow)-\alpha^2(T^\leftarrow)) q' \right)}
 e^{2\pi i\kappa \eta O(1)}F^{(0,0)}_{T^\rightarrow,T^\leftarrow}\overline{F^{(0,0)}_{T'^\rightarrow,T^\leftarrow}}\,,\nonumber
\end{eqnarray}
where $T^\leftarrow\in \mathcal{T}^\leftarrow$ and $ T^\rightarrow,T'^\rightarrow\in \mathcal{T}^\rightarrow$. The symbol $\widehat\sum$ in the first formula indicates that the summation is done over the resonant configurations, i.e., those for which \eqref{global} holds.\smallskip

\noindent{\it Proof of the Theorem \ref{tg2}.} In the wave packet decomposition \eqref{wpack1}, we organize summation in parameter $\ell: |\ell|\le \delta \kappa$ dyadically and write $f_o=\sum_{j\le N}f_j$, $N= \log_2(\delta \kappa)$ and each $f_j$ corresponds to the the range $|\ell|\sim \ell_j=2^{j-N}\delta\kappa$, i.e., $2^{j-N-1}\delta\kappa<|\ell|\le 2^{j-N}\delta\kappa$. Specifically, 
$
f_j(x)=\sum_{n\in \Z}\sum_{|\ell|\sim 
\ell_j} \omega_{n,\ell}(x,k)f_{n,\ell}(\eta)\,.
$
From \eqref{ss1}, we get
\begin{equation}\label{pogh1}
\|f_o\|_2^2\sim \sum_{j}\|f_j\|_2^2, \quad \|f_j\|_2\lesssim \|f_o\|_2\,.
\end{equation}
Applying the triangle inequality and then Cauchy-Schwarz bound, 
one can write
\begin{equation}\label{nov7}
\|Qf_o\|_{2,2}^2\le \left(\sum_{j\le N}\|Qf_j\|_{2,2}\right)^2\lesssim N \sum_{j\le N}\|Qf_j\|_{2,2}^2   \,, 
\end{equation}
where $\|g(x,\eta)\|_{2,2}$ for a function $g$ refers to the norm $\left(\int_{I^{-1}}\int_{\R}|g(x,\eta)|^2dxd\eta\right)^{\frac 12}$. The size of $N$ will be negligible in the later estimates so we will focus on estimating each term in the sum.

 Consider $f_j$ that corresponds to the dyadic shell $|\ell|, |\ell'|\sim \ell_j$ and focus on the corresponding expression 
 $
 \sum_{T^\leftarrow\in \mathcal{T}^\leftarrow}\Bigl(\Sigma^{(nr)}_{T^\leftarrow}+\Sigma^{(r)}_{T^\leftarrow}\Bigr)
 $
 (see \eqref{lhl2}), where $f_j$ is used instead of $f_o$ and therefore the tubes $T^\rightarrow$ and $T'^\rightarrow$  satisfy $|\alpha(T^\rightarrow)|\sim \ell_j/\kappa, |\alpha(T'^\rightarrow)|\sim \ell_j/\kappa$. The contribution from the non-resonant terms is negligible so we focus only on the resonant terms collected in $A_j$, which is defined as in \eqref{kook1}. Using a trivial bound $ab\le (a^2+b^2)/2$, we get
\begin{eqnarray}\label{kook2}
|A_j|\lesssim A_j^{(1)}+A_j^{(2)},\\
 A_j^{(1)}=\sum_{T^\rightarrow } |f^*_{T^\rightarrow}|^2\sum_{T'^\rightarrow } 
 \sum_{T^\leftarrow\in \mathcal{G}(T^\rightarrow,T'^\rightarrow)}
 |F^{(0,0)}_{T'^\rightarrow,T^\leftarrow}|^2,\\
 A_j^{(2)}=\sum_{T'^\rightarrow } |f^*_{T'^\rightarrow}|^2\sum_{T^\rightarrow } 
 \sum_{T^\leftarrow\in \mathcal{G}(T^\rightarrow,T'^\rightarrow)}
 |F^{(0,0)}_{T^\rightarrow,T^\leftarrow}|^2\,.
\end{eqnarray}
We will estimate $A_j^{(1)}$, the bound for $A_j^{(2)}$ is similar. From \eqref{lil1}, one has
\[
f^*_{n,\ell}(\eta)=\frac{\eta}{ 2\pi \sqrt\kappa}\int_{0}^{4\pi \kappa}f_j(s+2\pi n\kappa)h(s\kappa^{-1})e^{- i\ell \eta s\kappa^{-1}} ds\,.
\]
We write $\xi=\ell \eta/\kappa$ and recall that $|\ell|\sim \ell_j$. Then,
\begin{equation}\label{pobg1}
\int_{I^{-1}}A_j^{(1)}d\eta\lesssim 
\frac{1}{\ell_j} \sum_n \int_{|\xi| \kappa/\ell_j\in 2I^{-1}}\left| \int_{0}^{4\pi \kappa}f_j(s+2\pi n\kappa)h(s\kappa^{-1})e^{- i\xi s} ds\right|^2 M_{n}(\xi)d\xi\,,
\end{equation}
where, with fixed $n$,
\[
M_n(\xi):= \sum_{\ell} \sum_{T'^\rightarrow } 
 \sum_{T^\leftarrow\in \mathcal{G}(T_{n,\ell}^\rightarrow,T'^\rightarrow)}
 |F^{(0,0)}_{T'^\rightarrow,T^\leftarrow}(\xi\kappa/\ell)|^2
\]
and we again emphasize  that $|\ell|\sim \ell_j, |\ell'|\sim \ell_j$ in the sum above. Recall that the symbol $R$ denotes a positive quantity that satisfies $R\lessapprox 1$.
If we can show that 
\begin{equation}
\label{nv1}
\sup_{|\xi|\kappa/\ell_j\in 2I^{-1}}M_n(\xi)\lessapprox  (\ell_j\kappa^{-\frac 12}+1)T^{1-2\gamma}(\ell_j+R)\kappa, \quad 
\end{equation}
then the Plancherel theorem  would yield (after we extend integration in $\xi$ over the whole real line $\xi\in \R$ in the right-hand side of \eqref{pobg1}):
\begin{equation}\label{nov8}
\int_{I^{-1}}A_j^{(1)}d\eta\lessapprox  \frac{(\ell_j\kappa^{-\frac 12}+1)T^{1-2\gamma}(\ell_j+O(R))\kappa}{\ell_j}\|f_j\|_2^2\,.
\end{equation}
We focus on proving \eqref{nv1} now, where we can choose  $n=0$ without loss of generality. Consider the expression
\begin{equation}\label{comfi}
\sum_{|\ell|\sim \ell_j} \sum_{T'^\rightarrow } 
 \sum_{T^\leftarrow\in \mathcal{G}(T_{0,\ell}^\rightarrow,T'^\rightarrow)}
 |F^{(0,0)}_{T'^\rightarrow,T^\leftarrow}(\xi\kappa/\ell)|^2
\end{equation}
that defines $M_n$.
Here, all tubes $T'^\rightarrow$ satisfy condition $|\ell'|\sim \ell_j$.

\begin{picture}(200,150)
\linethickness{0.2mm}
\put(5,45){\color{brown}\vector(1,-0.1){200}}
\put(5,90){\color{brown}\vector(5,-1){250}}

\put(5,5){\vector(0,1){120}}

\put(5,90){\vector(1,0){250}}
\put(250,80){$t$}

\put(132,29){\line(0,1){61}}

\put(160,29){\line(0,1){61}}

\put(132,32){\line(1,0){29}}

\put(135,20){$B'$}

\put(148,65){$B$}

\put(-5,90){$0$}

\put(-5,120){$x$}

\put(70,25){$T'^\rightarrow$}
\put(230,5){Figure 4}
\put(70,81){$T^\rightarrow$}
\put(195,80){$T^\leftarrow$}
\put(210,110){\color{brown}\vector(-1,-1){100}}


\put(130,30){\textcolor{red}{\rule{4\unitlength}{4\unitlength}}}

\put(158,57){\textcolor{red}{\rule{4\unitlength}{4\unitlength}}}

\end{picture}

We can organize \eqref{comfi} by first summing over all cubes $B'$ that are at the intersection of $T'^\rightarrow,T^\leftarrow$, i.e., $B'\cap (T'^\rightarrow\cap T^\leftarrow)\neq \emptyset$. By assumptions that we made in subsection \ref{spaspa}, this cube is unique if it exists. Hence,
\begin{equation}\label{comfi1}
\sum_{|\ell|\sim \ell_j} \sum_{T'^\rightarrow } 
 \sum_{T^\leftarrow\in \mathcal{G}(T_{0,\ell}^\rightarrow,T'^\rightarrow)}
 |F^{(0,0)}_{T'^\rightarrow,T^\leftarrow}(\xi\kappa/\ell)|^2=\sum_{B'}\sum_{|\ell|\sim \ell_j} \sum_{T'^\rightarrow } 
 \sum_{\stackrel{T^\leftarrow\in \mathcal{G}(T_{0,\ell}^\rightarrow,T'^\rightarrow):}{ B'\cap (T'^\rightarrow\cap T^\leftarrow)\neq \emptyset}}
 |F^{(0,0)}_{T'^\rightarrow,T^\leftarrow}(\xi\kappa/\ell)|^2\,.
\end{equation}

Fix any $B'$. If $n'$ in $T'^\rightarrow=T^\rightarrow_{n',\ell'}$ is given, it defines $\ell'$ essentially uniquely since the cube $B'$ is at least $c\kappa$ units away from the axis $OX$. We will show that for each $B'$, one has
\begin{equation}\label{nov5}
\sum_{|\ell|\sim \ell_j} \sum_{T'^\rightarrow } 
  \sum_{\stackrel{T^\leftarrow\in \mathcal{G}(T_{0,\ell}^\rightarrow,T'^\rightarrow):}{ B'\cap (T'^\rightarrow\cap T^\leftarrow)\neq \emptyset}}
 |F^{(0,0)}_{T'^\rightarrow,T^\leftarrow}(\xi\kappa/\ell)|^2\lessapprox (\ell_j\kappa^{-\frac 12}+1)T^{1-2\gamma}. 
\end{equation}
Let $B: B\cap(T^\rightarrow_{0,\ell}\cap T^\leftarrow)\neq \emptyset$ be the other cube from the global resonance condition between $T^\rightarrow_{0,\ell}, T'^\rightarrow$, and $ T^\leftarrow$.
Recall that if the ``unit coordinates'' of $B$ are $(p,q)$, the actual coordinates of the southwestern corner of $B$ are $x=2\pi p\kappa, t=2\pi q\kappa$.
Given that $|\ell|\sim \ell_j$, the ``unit coordinates" $(p,q)$ of $B$ satisfy $|p|\sim \ell_j, q\sim \kappa$. Then,  the total number of all possible cubes $B$ satisfies $\#B\lesssim \ell_j\kappa$. The global resonance condition \eqref{global} gives
\[
O(R)=\frac{\alpha^2(T^\rightarrow)-\alpha^2(T^\leftarrow)}{\alpha^2(T'^\rightarrow)-\alpha^2(T^\leftarrow)}q-q'=q-q'+q\frac{\alpha^2(T^\rightarrow)-\alpha^2(T'^\rightarrow)}{\alpha^2(T'^\rightarrow)-\alpha^2(T^\leftarrow)}
\]
and it shows that $|q-q'|\lesssim \ell_j+R$ because $|\alpha(T^\rightarrow)|\lesssim \ell_j/\kappa$ and $|\alpha(T'^\rightarrow)|\lesssim \ell_j/\kappa$. Therefore, the cubes $B$ and $B'$ that are involved in the global resonance are at most $C(\ell_j+R)$ units away from each other. Hence, we get $\# B'\lesssim (\ell_j+R)\kappa$ for the number of cubes $B'$ and 
\eqref{nov5} implies  \eqref{nv1}. 
To prove \eqref{nov5}, we will show that
\begin{equation}\label{nv6}
\sum_{|\ell|\sim \ell_j} \sum_{T'^\rightarrow } 
 \sum_{\stackrel{T^\leftarrow\in \mathcal{G}(T_{0,\ell}^\rightarrow,T'^\rightarrow):}{ B'\cap (T'^\rightarrow\cap T^\leftarrow)\neq \emptyset}}
 |F^{(0,0)}_{T'^\rightarrow,T^\leftarrow}(\xi\kappa/\ell)|^2\lesssim (\ell_j\kappa^{-\frac 12}+1) \int |\widehat{V}_{B'}(s_1,s_2)|^2ds_1ds_2
\end{equation}
for every $B'$.
Consider \eqref{gas1}. On the Fourier side, cover the  annulus $\{\rho_1-\epsilon<|\xi|<\rho_2+\epsilon\}$ that essentially supports $\widehat V_{B'}$ (see \eqref{ryr8}) by at most $C'\kappa^2$ balls or radius $\kappa^{-1}$ such that no point in that annulus belongs to more than $C$ such balls.
By \eqref{gas1}, every choice of $\ell, n'$ and $T^\leftarrow$ gives an evaluation of $|\widehat V_{B'}|\ast |\psi_\kappa|$ at a particular frequency which is at the intersection of at most $C$  balls. Since the value of $(|\widehat V_{B'}|\ast |\psi_\kappa|)^2$ at any point is essentially dominated by the average of $|\widehat V_{B'}|^2$ over  its $\kappa^{-1}$--neighborhood (see \eqref{tri2} for the precise bound), we only need to show that every such ball can correspond to at most $\lessapprox \ell_j\kappa^{-\frac 12}+1$ such triples $(\ell,n',T^\leftarrow)$ (recall here that $n'$ determines $\ell'$ essentially uniquely). 

We will use notation $\alpha=\alpha(T^\rightarrow)=\ell/\kappa$, $\alpha'=\alpha(T'^\rightarrow)=\ell'/\kappa$, and $\beta=\alpha(T^\leftarrow)$.
Checking the argument on the right-hand side of \eqref{gas1}, we obtain
\[
\xi \frac{\kappa(\alpha'-\beta)}{\ell}=\xi_1+O(\kappa^{-1}), \quad \xi \frac{\kappa(\alpha'^2-\beta^2)}{\ell}=\xi_2+O(\kappa^{-1})\,,
\]
where $(-\xi_1,\xi_2)$ is the center of the fixed ball in a partition and the bounds in \eqref{san1} yield  $|\xi_{1(2)}|\sim 1$. Recall that $|\xi|\sim \ell_j/\kappa$, $|\alpha|\sim |\alpha'|\sim \ell_j/\kappa$ and $|\beta|\sim 1$. Let $\mu_1=0.5\xi_1\xi^{-1}$ and $\mu_2=0.5\xi_2\xi^{-1}_1$. Then,
\begin{eqnarray}\label{nv3}
\alpha'=\frac 12\left(\xi^{-1}\xi_1\alpha+\frac{\xi_2}{\xi_1} \right)+O(\kappa^{-1})=\mu_1\alpha+\mu_2+O(\kappa^{-1}),\\ \beta=\frac 12\left(\frac{\xi_2}{\xi_1}-\xi^{-1}\xi_1\alpha\right)+O(\kappa^{-1})=\mu_2-\mu_1\alpha+O(\kappa^{-1})\label{nv4}
\end{eqnarray}
and we get $|\mu_1|\sim |\xi|^{-1}\sim \kappa/\ell_j, |\mu_2|\sim 1$.  These formulas  show that each $\alpha$ defines $\alpha'=\ell'/\kappa$ and $\beta$ up to $O(\kappa^{-1})$ correction, i.e., essentially uniquely. Now, we will employ the global resonance condition to get the restrictions on the possible range of $\alpha$.

Let $B'$ has ``unit coordinates'' $q'=\kappa$ and $p'=2h\kappa$ (all other values of $q'$ can be handled similarly). We already saw that $|h|\kappa\lesssim \ell_j+R$.  Recall that the slope of each tube $T^\rightarrow$ equals  $2\alpha(T^\rightarrow)$. Then, if $B$ has ``unit coordinates'' $(p,q)$, we can define  $t:=q-\kappa$ and use elementary geometric considerations (check Figure~4) to get
\[p=2\alpha (\kappa+t)+O(1)=2h\kappa+2t\beta+O(1)\,.
\]
So, $t={\kappa(\alpha-h)}/{(\beta-\alpha)}+O(1)$ and 
\[
q=\frac{\kappa(\beta-h)}{\beta-\alpha}+O(1)\,.
\]
We substitute that expression, along with formulas \eqref{nv3} and \eqref{nv4} for $\alpha'$ and $\beta$, into the global resonance condition \eqref{global}:
\[
(\alpha^2-\beta^2)q=(\alpha'^2-\beta^2)q'+O(R)\,,
\]
where $q'=\kappa$.
That gives an equation for finding $\alpha$:
\[
H(\alpha):=\alpha^2\mu_1(\mu_1-1)-\alpha(h(1-\mu_1)-\mu_2-2\mu_1\mu_2)+\mu_2(\mu_2-h)=O(R/\kappa)\,.
\]
We recall that $|\mu_1|\sim \kappa/\ell_j\gg 1,|\mu_2|\sim 1$ and $|\alpha|\sim \ell_j/\kappa$.
The function $H/(\mu_1(\mu_1-1))$ is a quadratic polynomial and its leading coefficient is equal to one. So, the allowed interval for solution $\alpha$ is at most $\ell_jR^{\frac 12}\kappa^{-\frac 32}$ in length, which comes from the case when the   roots of the parabola are close to each other (see Figure~5).

\begin{picture}(200,150)
\linethickness{0.2mm}
\qbezier(40,100)(80,0)(120,100)
\put(150,5){\vector(0,1){120}}
\put(5,54){\vector(1,0){300}}

\put(310,50){$\alpha$}
\put(10,120){$H(\alpha)/(\mu_1(\mu_1-1))$}
\put(5,50){\color{red}\line(1,0){300}}
\put(5,58){\color{red}\line(1,0){300}}



\put(230,5){Figure 5}




\end{picture}

 By \eqref{nv3} and \eqref{nv4}, each $\alpha$ determines $\alpha'$ and $\beta$ up to $O(\kappa^{-1})$ correction, i.e., essentially uniquely.
Since $\alpha=\ell/\kappa$ and $\ell\in \Z$, that gives us $\lessapprox \ell_j\kappa^{-\frac 12}+1$ choices for the number of triples $(\ell,T'^\rightarrow,T^\leftarrow)$ and \eqref{nv6} is proved. From our previous arguments, we know that \eqref{nv6} implies \eqref{nov8}.
Now, substitution of \eqref{nov8} into \eqref{nov7} yields
\[
\|Qf_o\|^2_{2,2}\lessapprox \sum_{j\le N} \frac{(\ell_j\kappa^{-\frac 12}+1)T^{1-2\gamma}(\ell_j+O(R))\kappa}{\ell_j}\|f_j\|^2\stackrel{\eqref{pogh1}}{\lessapprox} \kappa^{-\frac 32}T^{2-2\gamma}\|f_o\|^2 \sum_{j\le N}\ell_j\lesssim  \delta\kappa^{-\frac 12}T^{2-2\gamma}\|f_o\|^2\,.
\]
The proof of Theorem \ref{tg2} is finished.
\qed\medskip

Now, we are ready to prove our main result.

\noindent { \it Proof of Theorem \ref{tg1}.} Given our assumptions,
$U(0,t,k)f=U(0,t,k)f_o+o(1)$ uniformly in $t$ and $k$.
Take $t_j=2\pi jT/N$ where $j\in \{0,\ldots, N\}$ and $N$ is to be chosen later. Write $u:=U(0,t,k)f_o$ and define $\varepsilon_j$ through the formula
$
U(0,t_j,k)f_o=e^{ik\Delta t_j}f_o+\varepsilon_j(k)\,.
$
By the group property, $U(0,t_{j+1},k)=U(t_j,t_{j+1},k)U(0,t_{j},k)$. We can use the Duhamel expansion \eqref{duz1} for the first factor to get
\[
U(0,t_{j+1},k)f_o=e^{ik\Delta t_{j+1}}f_o-i\int_{t_j}^{t_{j+1}}e^{ik\Delta (t_{j+1}-\tau)}V(\cdot,\tau)e^{ik\Delta \tau }f_od\tau+U(t_j,t_{j+1},k)\varepsilon_j(k)+\Delta_{j}\,,
\]
where
\[
\Delta_{j}:=-\int_{t_j}^{t_{j+1}} e^{ik\Delta (t_{j+1}-\tau_1)}V(\cdot,\tau_1)\int_{t_j}^{\tau_1} e^{ik\Delta(\tau_1-\tau_2)}V(\cdot,\tau_2)u(\cdot,\tau_2,k)d\tau_2d\tau_1
\]
and
$
\|\Delta_{j}\|\lesssim (T/N)^2T^{-2\gamma}
$ because $\|u(\cdot,t,k)\|\lesssim  1$ for all $t$.
Hence, 
\[
\epsilon_{j+1}(k)=-i\int_{t_j}^{t_{j+1}}e^{ik\Delta (t_{j+1}-\tau)}V(\cdot,\tau)e^{ik\Delta \tau }f_od\tau+U(t_j,t_{j+1},k)\varepsilon_j(k)+\Delta_j\,.
\]
Consider the first term in the formula above. Following \eqref{ivan1}, we can write
\[
\chi_{t_j<t<t_{j+1}}\cdot V(x,t)=V^{(j)}(x,t)+V_{err}(x,t)\,,
\]
where 
\[
V^{(j)}:=\sum_{(p,q): 2B_{p,q}\subset \R\times [t_j,t_{j+1}]} V_{p,q}(x-2\pi \kappa p,t-2\pi \kappa q)\,.
\]
The term
$V_{err}(x,t)$ is supported in $t$ on $[t_j,t_j+C\kappa]\cup [t_{j+1}-C\kappa,t_{j+1}]$ and $\|V_{err}\|_{L^\infty(\R^2)}\lesssim T^{-\gamma}$. We write
\begin{eqnarray*}
\int_{t_j}^{t_{j+1}}e^{ik\Delta (t_{j+1}-\tau)}V^{(j)}e^{ik\Delta \tau }f_od\tau=\hspace{6cm}\\
\int_{0}^{t_{j+1}}e^{ik\Delta (t_{j+1}-\tau)}V^{(j)}e^{ik\Delta \tau }f_od\tau=e^{-ik\Delta(T-t_{j+1})}\int_{0}^{T}e^{ik\Delta (T-\tau)}V^{(j)}e^{ik\Delta \tau }f_od\tau
\end{eqnarray*}
and apply  Theorem \ref{tg2} to the 
\[
\int_{0}^{T}e^{ik\Delta (T-\tau)}V^{(j)}e^{ik\Delta \tau }f_od\tau
\]
recalling that $V^{(j)}(x,t)=0$ for $t<C_1T$. That gives
\[
\Bigl(\int_I\|\epsilon_{j+1}\|^2dk\Bigr)^{\frac 12}\lessapprox CT^{1-\gamma-\frac{1}{8}}+\left(\int_I\|\epsilon_{j}\|^2dk\right)^{\frac 12}+CT^{2-2\gamma}N^{-2}+CT^{\frac 12-\gamma}\,.
\]
Iterating $N$ times and using $\varepsilon_1=0$, we get
$
\Bigl(\int_I\|\epsilon_{N}\|^2dk\Bigr)^{\frac 12}\lessapprox NT^{1-\gamma-\frac{1}{8}}+T^{2-2\gamma}N^{-1}+NT^{\frac 12-\gamma}\,.
$
Choosing $N=T^{\frac{9}{16}-\frac{\gamma}{2}}$, one has
$
\Bigl(\int_I\|\epsilon_{N}\|^2dk\Bigr)^{\frac 12}\lessapprox T^{\frac{23}{16}-\frac{3\gamma}{2}}
$
and the proof is finished by letting \mbox{$\gamma_0=\frac{23}{24}\,.$}
\qed

\smallskip

\section{Creation of the resonances.}\label{s6}

In that section, we show how the free evolution can be distorted by the potential $V$ of a small uniform norm. In particular, we will see that the Corollary \ref{pl2} does not hold for $\gamma<1$. That explains that the set of resonant parameters in  Theorem \ref{tg1} can indeed be nonempty.\smallskip

\noindent {\bf Definition.} In our perturbation analysis, we will say that the solution $u$ in \eqref{e1} experiences the {\it anomalous dynamics} for a given  $T$-dependent initial data $f$ and $k$ if
$
\lim_{T\to\infty} \|u(x,T,k)-e^{ik\Delta T}f\|_2
$ either does not exist or is not equal to zero.

\begin{lemma}[{\bf Approximation Lemma}] Given real-valued $V(x,t)$ that satisfies $\|V(\cdot,t)\|_{L^\infty(\R)}\lesssim T^{-\gamma}, t\in[0,T]$, we suppose $N$ is chosen such that $N^{-1}T^{2-2\gamma}\lesssim  1$. Then,
\begin{equation}\label{e3}
\|U(0,t_j,k)-(e^{ik\Delta d}+Q_j)\cdot \ldots\cdot (e^{ik\Delta d}+Q_1)
\|\lesssim N^{-1}T^{2-2\gamma}, \,\, j=\{1,\ldots,N\}, 
\end{equation}
where $t_j:=jT/N, d:=T/N$, and 
$
Q_j:=-i\int_{t_{j-1}}^{t_j}e^{ik\Delta(t_j-\tau)}Ve^{ik\Delta(\tau-t_{j-1})}d\tau\,.
$
\end{lemma}
\begin{proof}
If
$
\Delta_j:=U(0,t_j,k)-(e^{ik\Delta d}+Q_j)\cdot \ldots\cdot (e^{ik\Delta d}+Q_1), \, R_j:=U(t_{j-1},t_j,k)-(e^{ik\Delta d}+Q_j)
$, then
\begin{eqnarray*}
U(0,t_j,k)=(e^{ik\Delta d}+Q_j+R_j)U(0,t_{j-1},k)\hspace{5cm}\\=(e^{ik\Delta d}+Q_j+R_j)((e^{ik\Delta d}+Q_{j-1})\cdot \ldots\cdot (e^{ik\Delta d}+Q_1)+\Delta_{j-1})\\
=(e^{ik\Delta d}+Q_{j})\cdot \ldots\cdot (e^{ik\Delta d}+Q_1)+R_jU(0,t_{j-1},k)+(e^{ik\Delta d}+Q_{j})\Delta_{j-1}\,.
\end{eqnarray*}
Hence,
$
\Delta_j=R_jU(0,t_{j-1},k)+(U(t_{j-1},t_j,k)-R_j)\Delta_{j-1}
$
and we use Lemma~\ref{l1} to get a bound
$
\|\Delta_j\|\le (1+\alpha)\|\Delta_{j-1}\|+\alpha, \, \|\Delta_1\|\le \alpha\,,
$
where $\alpha:=CT^{-2\gamma}d^2$. 
 Then, by induction,
$
\|\Delta_j\|\le (1+\alpha)^j-1\le e^{\alpha j}-1\lesssim \alpha j
$ and the last estimate holds provided that $\alpha j\lesssim  1$. Taking $j=N$, we get the statement of our lemma.
\end{proof}

 Taking $N=T^{2-2\gamma}\mu(T), \lim_{T\to\infty}\mu(T)=+\infty$ in this lemma, we get a good approximation for the dynamics by a product of $N$ relatively simple factors. Clearly, the same result holds if we replace $x\in \R$ by $x\in \R\backslash (CT)\Z$. \smallskip

(A) {\it Creation of anomalous dynamics using a bound state. } Consider any smooth non-positive  function $q(s)$ supported on $[-1,1]$ which is not equal to zero identically. The standard variational principle yields the existence of a positive bound state $\varphi(s)$ that solves $-\varphi''+\lambda q\varphi=E\varphi, E<0$ when a positive coupling constant $\lambda$ is large enough. The number $E$ is the smallest eigenvalue and 
$\varphi$ is an exponentially decaying smooth function. We can normalize it as $\|\varphi\|_{L^2(\R)}=1$. Hence, one gets
\[
i y_\tau=-\Delta y+\lambda qy, \quad y(s,\tau)=\varphi(s)e^{-iE\tau},\quad  y(s,0)=\varphi(s)\,.
\]
Given $T$ and $\alpha_T$, we can rescale the variables $u(x,t):=\alpha_T^{\frac 12} y(\alpha_T x,\alpha_T^2 t)$. Then, 
\[
iu_t=-\Delta u+\lambda \alpha_T^2q(\alpha_T x)u, \quad \|u(x,0)\|_{L^2(\R)}=1.
\]
The initial value $u(x,0)$ is now supported around the origin at scale $\alpha_T^{-1}$, and its Fourier transform is supported around the origin on scale $ \alpha_T$. The potential $V_2=\lambda\alpha_T^2q(\alpha_T x)$ satisfies $|V_2|\lesssim  \lambda \alpha_T^2$. Hence, taking $\alpha_T=T^{-\frac{\gamma}{2}}$, we satisfy $|V|\lesssim \lambda T^{-\gamma}$. Nonetheless, the function $u(x,0)$ of scale $T^{\gamma/2}\ll T^{1/2}$ evolves into $u(x,T)$ which is supported around zero and has the same scale as $u(x,0)$. On the other hand, $e^{i\Delta T}u(x,0)$ is supported around the origin and has the scale $T^{1-\frac{\gamma}{2}}\gg T^{\frac{\gamma}{2}}$ if $\gamma<1$. This is one example of anomalous transfer when the potential, small in the uniform norm, traps the wave and prevents its propagation.\smallskip

(B) {\it Anomalous dynamics: construction of a  resonance.} We now focus on another problem with time-independent potential:
\begin{equation}\label{resg1}
iu_t=-\Delta u+T^{-\gamma}\cos (2x)  u, \quad u(x,0)=T^{-\frac 12}(a_0e^{ix}+b_0e^{-ix}), \quad x\in \R/(T\Z), \quad T\in 2\pi \N\,,
\end{equation}
$\gamma\in (\frac 12,1)$ and $|a_0|^2+|b_0|^2=1$. Consider operators $Q_j$ from Approximation Lemma. They do not depend on $j$ and, if we denote $Q:=Q_j$, then
\[
T^{-\frac 12}Q(a_0e^{ix}+b_0e^{-ix})=-iT^{-\gamma-\frac 12}\int_{0}^{d} e^{i\Delta (d-\tau)
}\cos (2x) \cdot e^{i\Delta \tau}(a_0e^{ix}+b_0e^{-ix}) d\tau\,.
\]
For arbitrary $a$ and $b$, one can write
\begin{eqnarray*}
-T^{-\frac 12}Q(ae^{ix}+be^{-ix})=iT^{-\gamma-\frac 12}\int_{0}^{d} e^{i\Delta (d-\tau)
}\cos (2x) (ae^{i(x-\tau)}+be^{-i(x+\tau)})d\tau=\\
i\frac{a}{2}T^{-\gamma-\frac 12} \int_0^d e^{-i\tau}e^{i\Delta(d-\tau)}e^{3ix}d\tau+i\frac{b}{2}T^{-\gamma-\frac 12} \int_0^d e^{-i\tau}e^{i\Delta(d-\tau)}e^{-3ix}d\tau+\\
i\frac{a}{2}T^{-\gamma-\frac 12} \int_0^d e^{-i\tau}e^{i\Delta(d-\tau)}e^{-ix}d\tau+i\frac{b}{2}T^{-\gamma-\frac 12} \int_0^d e^{-i\tau}e^{i\Delta(d-\tau)}e^{ix}d\tau\,.
\end{eqnarray*}
We can arrange $d$ to  make sure that $d=T/N\in 2\pi \N$ and  $N\sim  T^{2-2\gamma}$. Then, 
\[
T^{-\frac 12}Q(ae^{ix}+be^{-ix})=\left(-\frac{ib}{2}T^{-\gamma-\frac 12}d\right)e^{ix}+\left(-\frac{ia}{2}T^{-\gamma-\frac 12}d\right)e^{-ix}
\]
and the product in  \eqref{e3} takes the form
\[
\left((e^{i\Delta d}+Q_j)\cdot \ldots\cdot (e^{i\Delta d}+Q_1)\right)(a_0e^{ix}+b_0e^{-ix})=a_je^{ix}+b_je^{-ix}\,,
\]
where
\[
\left(\begin{matrix}
a_{j}\\
b_{j}
\end{matrix}\right)=\left(\begin{matrix}
1& \lambda\\
\lambda&1
\end{matrix}\right)^j\left(\begin{matrix}
a_{0}\\
b_{0}
\end{matrix}\right), \quad \lambda= -\frac{i dT^{-\gamma}}{2}\,.
\]
From the Approximation Lemma, we get
\[
\|u(x,d j)-T^{-\frac 12}(a_je^{ix}+b_je^{-ix})\|_2=o(1), \quad T\to \infty
\]
for $j=o(N)$. The eigenvalues of the matrix $\left(\begin{smallmatrix} 1&\lambda\\\lambda&1\end{smallmatrix}\right)$ are $z_{\pm}=1\mp i|\lambda|$ so 
\[
\left(\begin{matrix}
a_{j}\\
b_{j}
\end{matrix}\right)=z_+^j\left(\begin{matrix}
1\\
1
\end{matrix}\right)\frac{a_0+b_0}{2}+z_-^j\left(\begin{matrix}
1\\
-1
\end{matrix}\right)\frac{a_0-b_0}{2}\,.
\]
Since $N\sim  T^{2(1-\gamma)}$,
$
z_+^j=e^{\log(1-i|\lambda|)j}=e^{-i(|\lambda|j+O(|\lambda|^2j))}=e^{-i|\lambda|j}(1+O(|\lambda|^2j))
$
if $j=o(N)$. A similar calculation can be done for $z_-^j$. Hence, already for $j\gg NT^{\gamma-1}$, the solution $u$ with initial data $e^{ix}T^{-\frac 12}$ will carry nontrivial $L^2$ norm on frequency band  $\{e^{i\alpha x}, |\alpha-1|>0.1\}$ which indicates the anomalous dynamics. That is achieved by the creation of a resonance that changes the direction of the wave. \smallskip

{\noindent \bf Remark.} We notice that introduction of the parameter $k>0$ in the form
\[
iu_t=-k\Delta u+T^{-\gamma}\cos (2x) u,\quad  u(x,0,k)=T^{-\frac 12}e^{ix}
\]
only rescales the time and potential, and  the resonance occurs for all positive $k$.  That also indicates that $\delta$ in \eqref{asd2} must be taken sufficiently small for \eqref{lop2} to hold. \smallskip 

{\noindent \bf Remark.} In the evolution $iu_t=-\Delta u+V(x,t)u$, the energy
\[
E(t)=\int \left({|u_x|^2}+V|u|^2\right) dx
\]
satisfies $E'(t)=\int V_t(x,t)|u(x,t)|^2dx$ provided that $V$ is smooth in $t$. In particular, $E$ is a conserved quantity for time-independent $V$.
 For initial data $f=e^{ix}T^{-\frac 12}$ and $V=T^{-\gamma}\cos(2x)$, we get $E=1$. That, however, does not contradict the existence of a resonance since other functions also give rise to the same energy, e.g., $\widetilde f=e^{-ix}T^{-\frac 12}$. Hence, in our example above, we realize the   transfer of $L^2$--norm along the equi-energetic set.\smallskip

The resonance phenomenon described above in \eqref{resg1} where $x\in \R/ (T\Z)$ also takes place if we consider the same equation on $\R$ and the initial data is replaced by $e^{ix}T^{-\frac 12} \mu(x/T)$ where $\mu$ is a compactly supported smooth bump  satisfying $\mu(x)=1$ for $|x|<1$. Next, we introduce parameter $k$ and  consider the problem 
\begin{equation}\label{p1}
iy_t=-k\Delta y+T^{-\gamma}\cos (2x)  y, \quad y(x,0,k)=T^{-\frac 12}e^{i\frac{1}{2k} x}\mu(x/T), \quad x\in \R\,,
\end{equation}
which exhibits the resonance for $k=\frac 12$ as we just established.
We recast it
using the modulation scaling described in the Appendix. In particular, $\psi(x,t,k):=e^{-i(\frac{t}{4k}+\frac{x}{2k})}y(x+t,t,k)$ solves
\begin{equation}\label{p2}
i\psi_t=-k\Delta \psi+T^{-\gamma}\cos (2x+2t) \psi, \quad \psi(x,0,k)=T^{-\frac 12}\mu(x/T), \quad x\in \R\,.
\end{equation}
The solution to problem \eqref{p2} satisfies \eqref{lop2} and so it is non-resonant for generic $k$. Nevertheless, for $k=\frac 12$, it is resonant, Indeed, it has no local oscillation when $t=0$ but starts to oscillate like $e^{-2ix}$ locally when $t\gg T^{\gamma}$ giving a boost to the Sobolev norms.  In particular,  the Corollary \ref{pl2} does not hold for $\gamma<1$ and the set $Res$ of resonant parameters $k$ in  Theorem~\ref{tg1} can indeed be nonempty.\bigskip

(C) {\it The lower bound for the norm of the one-collision operator. } For $x\in \R\backslash T \Z, T=2\pi N, N\in 2\mathbb{N}$, consider a one-collision operator $Q$ defined in \eqref{coll} and write it in the form:
\begin{equation}\label{lsf1}
Q=e^{ik\Delta (2\pi N)}Q_*, \quad Q_*:=\int_0^{2\pi N} e^{-ik\Delta t}V(\cdot,t)e^{ik\Delta t}dt\,.
\end{equation}
Now, we study the operator norm $\|Q(k)\|$ as a function in $k\in I$. 

\begin{lemma}Let $I\subset \R^+$ be any interval and $\gamma\in (0,1)$. There is $V$ such that $\|V\|_{L^\infty(\R\backslash T\Z\times [0,T])}\le CT^{-\gamma}$ such that 
\[
\int_I \|Q(k)\|^2dk\gtrsim T^{1-\gamma}\,.
\]
\end{lemma}
\begin{proof}
Since $e^{ik\Delta (2\pi N)}$ is unitary, we only need to focus on $\|Q_*(k)\|$ from \eqref{lsf1}.
Write the matrix representation of $Q_*$ in the Fourier orthonormal basis $\{T^{-\frac 12}e^{ixn/N}, n\in \Z\}$:

\[
q^\ast_{n,m}(k)=\int_0^{2\pi N}\widehat V_{n-m}(t)e^{ik(n^2-m^2)tN^{-2}}dt\,,
\]
where $\widehat V$ denotes the Fourier transform of $V$ in the $x$-variable. Take $V=N^{-\gamma}e^{i(x-t)}$, which yields
$
n-m=N
$
and
\[
|q^\ast_{n,n-N}(k)|\sim N^{-\gamma}\left|\frac{1-e^{2\pi i N(-1+k(1+2mN^{-1}))}}{-1+k(1+2mN^{-1})}\right|\,.
\]
Since $\|Q^*\|\ge \max_{n,m}|q^*_{n,m}|> CN^{1-\gamma}$ for $k\in I'\subset I$, $|I'|>c|I|$ with some positive $c$, we get our statement. 
\end{proof}

\noindent {\bf Remark.} Notice that potential $V$ in the proof above oscillates and its Fourier support is separated from the origin. The norm $\|Q(k)\|$ in this example is large when $2mN^{-1}+1-k^{-1}$ is close to zero. Hence, if $I$ contains $k=1$, even restricting $m$ to the range $|m|\le \delta N, \delta\ll 1$ does not prohibit the norm to be of size $N^{1-\gamma}$. This construction can be easily adapted to $x\in \R$ by using our wave packet decomposition.

\smallskip

\section{Appendix.}

\noindent {\it   Basic properties of 1d Schr\"odinger evolution.} The following two scaling properties of Schr\"odinger evolution  can be checked by direct inspection.\smallskip

{\it (a) Modulation.} If $u(x,t,k)$ solves 
\[
 iu_t=-k\Delta u+q(x,t)u, \quad u(x,0,k)=f(x)\,,
\] 
 then $\psi(x,t,k):=e^{-i\frac{\beta^2}{4k}t-i\frac{\beta}{2k}x}u(x+\beta t,t,k)$ solves
 \begin{equation}\label{isad1}
 i\psi_t=-k\Delta \psi+q(x+\beta t,t)\psi, \quad \psi(x,0,k)=e^{-i\frac{\beta}{2k}x}f(x).
 \end{equation}\smallskip

{\it (b) Scaling of time and space variables.} If $u(x,t)$ solves 
\[
 iu_t=-\Delta u+q(x,t)u, \quad u(x,0)=f(x)\,,
\] 
 then $\phi(x,t):= u(\beta x,\sigma t)$ solves
 \[
 i\phi_t=-\sigma\beta^{-2}\Delta \phi+\sigma q(\beta x,\sigma t)\phi, \quad \phi(x,0)=f(\beta x)
 \]
for all $\beta\neq 0$ and $\sigma>0$.\smallskip

{\it (c) Evolution of a bump function.} Suppose $h\in \mathbb{S}(\R)$ and let 
\begin{equation}
\mathcal{U}(x,t)= e^{it\Delta}h\,.\label{gene1}
\end{equation} Then, 
$
\langle x\rangle^\beta|(\partial_t^\nu\partial_x^\alpha \mathcal{U})(x,t)|<C_{\alpha,\beta,\nu},\,\, \alpha,\beta,\nu\in \Z^+
$
uniformly in $t\in [-t_0,t_0]$ with an arbitrary fixed positive $t_0$. That is immediate from the Fourier representation of Schr\"odinger evolution.\smallskip

{\it (d) Evolution of a scaled bump.} 
Suppose $h\in \mathbb{S}(\R)$ and $\kappa=\sqrt T\ge 1$. Then, 
\begin{equation}\label{ffa1}
e^{it\Delta}h(x/\kappa)=\mathcal{U}(x/\kappa,t/\kappa^2)
\end{equation}
and hence
\[
\partial_t^\nu\partial_x^\alpha \Bigl(e^{it\Delta}h(x/\kappa)\Bigr)=\kappa^{-\alpha-2\nu}(\partial_t^\nu\partial_x^\alpha \mathcal{U})(x/\kappa,t/\kappa^2),\quad \alpha,\nu\in \Z^+\,.
\]
That follows directly from the previous observation. For $\alpha=\nu=0$ and $\kappa\to\infty$, the function in the right-hand side is essentially supported in the neighborhood of the tube $(x,t)\in [-\kappa,\kappa]\times [-\kappa^2,\kappa^2]$, when $t$ is restricted to $[-T,T]$ but that localization is not exact as function $\mathcal{U}$ is not compactly supported in $x$. To have a sharper form of localization, we apply the following decomposition. Suppose $\phi\in C_c^\infty(\R), \, \supp \, \phi\in [-2\pi,2\pi]$ and 
$
1=\sum_{\lambda\in \Z}\phi(x-2\pi\lambda)\,.
$
Then, we can write
\begin{equation}\label{trans1}
\mathcal{U}(x,t)=\sum_{\lambda\in \Z} \mathcal{U}_\lambda(x,t), \quad \mathcal{U}_\lambda(x,t):=\mathcal{U}(x,t)\phi(x-2\pi\lambda)
\end{equation}
and
\begin{equation}\label{pio1}
\langle \lambda\rangle^\beta|\partial^\nu_t\partial^\alpha_x \mathcal{U}_\lambda(x,t)|\le C_{\alpha,\nu,\beta}, \quad \alpha,\beta,\nu\in \Z^+
\end{equation}
if $t\in [-1,1]$.
Hence, we can rewrite 
\begin{equation}
e^{it\Delta}\phi(x/\kappa)=\sum_{\lambda\in \Z} \mathcal{U}_\lambda(x/\kappa,t/\kappa^2)\,,
\end{equation}
where
each $\mathcal{U}_\lambda(x/\kappa,t/\kappa^2)$ is supported on the tube $[-2\pi\kappa +2\pi \lambda\kappa,2\pi\kappa+2\pi\lambda\kappa]\times [-T,T]$ when $t\in [-T,T]$. Moreover, \eqref{pio1} indicates that the contribution from $\mathcal{U}_\lambda(x/\kappa,t/\kappa^2)$ is negligible for large $\lambda$.
\bigskip

\bibliographystyle{plain}

\bibliography{Summer_bib}

\begin{thebibliography}{10}

\bibitem{bourgain1}
J.~Bourgain.
\newblock On growth of {S}obolev norms in linear {S}chr\"odinger equations with
  smooth time dependent potential.
\newblock {\em J. Anal. Math.}, 77:315--348, 1999.

\bibitem{bourgain}
J.~Bourgain.
\newblock On random {S}chr\"odinger operators on {$\Bbb Z^2$}.
\newblock {\em Discrete Contin. Dyn. Syst.}, 8(1):1--15, 2002.

\bibitem{ck1}
M.~Christ and A.~Kiselev.
\newblock Maximal functions associated to filtrations.
\newblock {\em J. Funct. Anal.}, 179(2):409--425, 2001.

\bibitem{ck2}
M.~Christ and A.~Kiselev.
\newblock Scattering and wave operators for one-dimensional {S}chr\"odinger
  operators with slowly decaying nonsmooth potentials.
\newblock {\em Geom. Funct. Anal.}, 12(6):1174--1234, 2002.

\bibitem{den}
S.A. Denisov.
\newblock Absolutely continuous spectrum of multidimensional {S}chr\"odinger
  operator.
\newblock {\em Int. Math. Res. Not.}, (74):3963--3982, 2004.

\bibitem{d1}
S.A. Denisov.
\newblock The generic behavior of solutions to some evolution equations:
  asymptotics and {S}obolev norms.
\newblock {\em Discrete Contin. Dyn. Syst.}, 30(1):77--113, 2011.

\bibitem{d2}
S.A. Denisov.
\newblock Multidimensional {$L^2$} conjecture: a survey.
\newblock In {\em Recent trends in analysis}, volume~16 of {\em Theta Ser. Adv.
  Math.}, pages 101--112. Theta, Bucharest, 2013.

\bibitem{d0}
S.A. Denisov.
\newblock The {S}obolev norms and localization on the {F}ourier side for
  solutions to some evolution equations.
\newblock {\em Comm. Partial Differential Equations}, 39(9):1635--1657, 2014.

\bibitem{guth}
X.~Du, L.~Guth, and X.~Li.
\newblock A sharp {S}chr\"odinger maximal estimate in {$\Bbb R^2$}.
\newblock {\em Ann. of Math. (2)}, 186(2):607--640, 2017.

\bibitem{yao}
L.~Erd\H{o}s, M.~Salmhofer, and H.~Yau.
\newblock Quantum diffusion of the random {S}chr\"odinger evolution in the
  scaling limit. {II}. {T}he recollision diagrams.
\newblock {\em Comm. Math. Phys.}, 271(1):1--53, 2007.

\bibitem{erdogan}
M.B. Erdogan, R.~Killip, and W.~Schlag.
\newblock Energy growth in {S}chr\"odinger's equation with {M}arkovian forcing.
\newblock {\em Comm. Math. Phys.}, 240(1-2):1--29, 2003.

\bibitem{safronov}
O.~Safronov.
\newblock Absolutely continuous spectrum of one random elliptic operator.
\newblock {\em J. Funct. Anal.}, 255(3):755--767, 2008.

\bibitem{wang}
W.-M. Wang.
\newblock Bounded {S}obolev norms for linear {S}chr\"odinger equations under
  resonant perturbations.
\newblock {\em J. Funct. Anal.}, 254(11):2926--2946, 2008.

\end{thebibliography}

\end{document}